\documentclass{amsproc}%
\usepackage{amsfonts}
\usepackage{amsmath}
\usepackage{amssymb}
\usepackage{graphicx}%
\providecommand{\U}[1]{\protect\rule{.1in}{.1in}}
\providecommand{\U}[1]{\protect\rule{.1in}{.1in}}
\providecommand{\U}[1]{\protect\rule{.1in}{.1in}}
\providecommand{\U}[1]{\protect\rule{.1in}{.1in}}
\providecommand{\U}[1]{\protect\rule{.1in}{.1in}}
\providecommand{\U}[1]{\protect\rule{.1in}{.1in}}
\providecommand{\U}[1]{\protect\rule{.1in}{.1in}}
\providecommand{\U}[1]{\protect\rule{.1in}{.1in}}
\theoremstyle{plain}

\newtheorem{corollary}{Corollary}

\newtheorem{definition}{Definition}
\newtheorem{example}{Example}

\newtheorem{lemma}{Lemma}

\newtheorem{problem}{Problem}
\newtheorem{proposition}{Proposition}
\newtheorem{remark}{Remark}

\newtheorem{theorem}{Theorem}
\numberwithin{equation}{section}
\begin{document}
\title[MF-traces]{MF-traces and a Lower Bound for the Topological Free Entropy Dimension in
Unital C*-algebras}
\author{Don Hadwin}
\address{University of New Hampshire}
\email{don@unh.edu}
\urladdr{http://www.math.unh.edu/\symbol{126}don }
\author{Qihui Li}
\address{East China University of Science and Technology, Shanghai}
\email{lqh991978@gmail.com}
\author{Weihua Li}
\address{Columbia College Chicago}
\email{whli@colum.edu}
\author{Junhao Shen}
\address{University of New Hampshire}
\email{jog2@math.unh.edu}
\thanks{This paper is in final form and no version of it will be submitted for
publication elsewhere.}

\begin{abstract}
We continue work on topological free entropy dimension $\delta_{\text{top}}$
in \cite{HS2}, \cite{HS3}, \cite{HLS}. We introduce the notions of MF-trace,
MF-ideal, and MF-nuclearity and use these concepts to obtain upper and lower
bounds for $\delta_{\text{top}}$, and in many cases we obtain an exact formula
for $\delta_{\text{top}}$. We also discuss semicontinuity properties of
$\delta_{\text{top}}$.

\end{abstract}
\maketitle

\section{Introduction}

This paper is a continuation of the work in \cite{HS2}, \cite{HS3}, \cite{HLS}
on D. Voiculescu's topological free entropy dimension $\delta_{\text{top}%
}\left(  x_{1},\ldots,x_{n}\right)  $ for an $n$-tuple $\vec{x}=\left(
x_{1},\ldots,x_{n}\right)  $ of elements in a unital C*-algebra. Our main
results concern the new concept of an $MF$\emph{-trace }on an\emph{ }$MF$
C*-algebra. We prove that the set of $MF$-traces is nonempty, convex and
weak$^{\ast}$-compact. We use $MF$-traces to obtain an important lower bound
for $\delta_{\text{top}}\left(  x_{1},\ldots,x_{n}\right)  $. In particular if
$C^{\ast}\left(  x_{1},\ldots,x_{n}\right)  $ either has no finite-dimensional
representations or infinitely many inequivalent irreducible representations,
then $\delta_{\text{top}}\left(  x_{1},\ldots,x_{n}\right)  \geq1$. We also
define the $MF$-ideal of an MF C*-algebra $\mathcal{A}$ to be the set
$\mathcal{J}_{MF}\left(  \mathcal{A}\right)  $ of all $x\in\mathcal{A}$ such
that $\tau\left(  x^{\ast}x\right)  =0$ for every $MF$-trace $\tau$. We show
that often $\delta_{\text{top}}$ depends on $\mathcal{A}/\mathcal{J}%
_{MF}\left(  \mathcal{A}\right)  $. Additionally, we define the notion of an
MF-nuclear C*-algebra and show that $\delta_{\text{top}}\left(  x_{1}%
,\ldots,x_{n}\right)  \leq1$ when $C^{\ast}\left(  x_{1},\ldots,x_{n}\right)
$ is MF-nuclear, greatly extending our previous result \cite{HLS} showing that
$\delta_{\text{top}}\left(  x_{1},\ldots,x_{n}\right)  \leq1$ when $C^{\ast
}\left(  x_{1},\ldots,x_{n}\right)  $ is nuclear. If $C^{\ast}\left(
x_{1},\ldots,x_{n}\right)  $ is MF-nuclear and residually finite-dimensional
(RFD), then
\[
\delta_{\text{top}}\left(  x_{1},\ldots,x_{n}\right)  =1-\frac{1}%
{\dim\mathcal{A}}.
\]
We also introduce and study two important classes of $MF$ algebras, and we
prove a semicontinuity result for $\delta_{\text{top}}$ restricted to the
second class of algebras.

The organization of the paper is as follows. In section 2 we recall
Voiculescu's definition \cite{DV5} of topological free entropy dimension and
previous results from \cite{HS2}, \cite{HS3}, \cite{HLS}. In section 3 we
introduce MF-traces and the MF-ideal. In section 4, we introduce the notion of
an MF-nuclear C*-algebra. In section 5 we prove a general lower bound for
$\delta_{\text{top}}\left(  x_{1},\ldots,x_{n}\right)  $. Finally, in section
6, we introduce two classes of MF C*-algebras: the class $\mathcal{S}$ of
those algebras for which every trace is an MF-trace, and the class
$\mathcal{W}$ of those C*-algebras whose MF-ideal is trivial. In this section
we also prove a semicontinuity result for $\delta_{\text{top}}\left(
x_{1},\ldots,x_{n}\right)  $ inside $\mathcal{S}\cap\mathcal{W}$, and we
provide examples that show that this semicontinuity generally fails without
severe restrictions.

\section{Definitions and Preliminaries}

In this section, we are going to recall Voiculescu's definition of the
topological free entropy dimension of $n$-tuples of elements in a unital C*-algebra.

We often use tuples like $\left(  y_{1},\ldots,y_{n}\right)  $ and use
$\vec{y}$ to denote them. We use the notation $\vec{y}$ and $\left(
y_{1},\ldots,y_{n}\right)  $ interchangeably. So $\vec{A}$ denotes $\left(
A_{1},\ldots,A_{n}\right)  $.

\subsection{Notation for the GNS Representation.}

Suppose $\tau$ is a tracial state on a unital C*-algebra $\mathcal{A}$. Then
there is a Hilbert space $H,$ a unit vector $e\in H,$ and a representation
$\pi_{\tau}:\mathcal{A}\rightarrow B\left(  H\right)  $ such that $\pi_{\tau
}\left(  \mathcal{A}\right)  e$ is dense in $H$ and, for every $a\in
\mathcal{A}$,%
\[
\tau\left(  a\right)  =\left(  \pi_{\tau}\left(  a\right)  e,e\right)  .
\]
We define the faithful normal trace $\hat{\tau}:\mathcal{\pi}_{\tau}\left(
\mathcal{A}\right)  ^{\prime\prime}\rightarrow\mathbb{C}$ by $\hat{\tau
}\left(  T\right)  =\left(  Te,e\right)  $.

\subsection{Covering Numbers and Box Dimension for a Metric Space.}

Suppose $(X,d)$ is a metric space and $K$ is a subset of $X.$ A family of
balls in $X$ is called a covering of $K$ if the union of these balls covers
$K$ and the centers of these balls lie in $K.$ If $\omega>0$, then the
covering number $\nu_{d}\left(  K,\omega\right)  $ is the smallest cardinality
of a covering of $K$ with $\omega$-balls. Equivalently, an $\omega$-net in $K$
is a subset $E\subseteq K$ such that, for every $x\in K$ there is an $e\in E$
such that $d\left(  x,e\right)  <\omega$. Then $\nu_{d}\left(  K,\omega
\right)  $ is the minimum cardinality of an $\omega$-net in $K$. The (upper)
\emph{box dimension }$\left(  \emph{Minkowski\ dimension}\right)  $\emph{
}of\emph{ }$K$ is defined as
\[
\dim_{\text{\textrm{box}}}\left(  K\right)  =\limsup_{\omega\rightarrow0^{+}%
}\frac{\log\nu_{d}\left(  K,\omega\right)  }{-\log\omega}.
\]
Here is a list of useful results. For an elementary account of these ideas see
\cite{DH}.

\begin{lemma}
\begin{enumerate}
\item If $\left\Vert \cdot\right\Vert $ is any norm on $\mathbb{R}^{k},$
$\omega>0$ and $B$ is the closed unit ball, then $\left(  \frac{1}{\omega
}\right)  ^{k}\leq\nu_{\left\Vert \cdot\right\Vert }\left(  B,\omega\right)
\leq\left(  \frac{3}{\omega}\right)  ^{k}.$

\item If $E$ is a bounded subset of $\mathbb{R}^{k}$ with positive Lebesgue
measure, then $\dim_{\text{\textrm{box}}}\left(  E\right)  =k$.

\item If $\mathcal{U}_{k}$ denotes the group of $k\times k$ unitary complex
matrices, and $\omega>0$, then%
\[
\left(  \frac{1}{\omega}\right)  ^{k^{2}}\leq\nu_{\left\Vert \cdot\right\Vert
}\left(  \mathcal{U}_{k},\omega\right)  \leq\left(  \frac{9\pi e}{\omega
}\right)  ^{k^{2}}.
\]

\end{enumerate}
\end{lemma}

\subsection{Covering Numbers in $(\mathcal{M}_{k}(\mathbb{C}))^{n}$}

Let $\mathcal{M}_{k}(\mathbb{C})$ be the $k\times k$ full matrix algebra with
entries in $\mathbb{C},$ and $\tau_{k}$ be the normalized trace on
$\mathcal{M}_{k}(\mathbb{C}),$ i.e., $\tau_{k}=\frac{1}{k}Tr,$ where $Tr$ is
the usual trace on $\mathcal{M}_{k}(\mathbb{C}).$ Let $\mathcal{U}_{k}$ denote
the group of all unitary matrices in $\mathcal{M}_{k}(\mathbb{C}).$ Let
$\left(  \mathcal{M}_{k}(\mathbb{C})\right)  ^{n}$ denote the direct sum of
$n$ copies of $\mathcal{M}_{k}(\mathbb{C}).$ Let $\mathcal{M}_{k}%
^{sa}(\mathbb{C})$ be the subalgebra of $\mathcal{M}_{k}(\mathbb{C})$
consisting of all selfadjoint matrices of $\mathcal{M}_{k}(\mathbb{C}).$ Let
$\left(  \mathcal{M}_{k}^{sa}(\mathbb{C})\right)  ^{n}$ be the direct sum (or
orthogonal sum) of $n$ copies of $\mathcal{M}_{k}^{sa}(\mathbb{C}).$ Let
$\left\Vert \cdot\right\Vert $ be an operator norm on $\mathcal{M}_{k}\left(
\mathbb{C}\right)  ^{n}$ defined by
\[
\left\Vert \left(  A_{1},\ldots,A_{n}\right)  \right\Vert =\max\left\{
\left\Vert A_{1}\right\Vert ,\ldots,\left\Vert A_{n}\right\Vert \right\}
\]
for all $\left(  A_{1},\ldots,A_{n}\right)  $ in $\mathcal{M}_{k}\left(
\mathbb{C}\right)  ^{n}.$ Let $\left\Vert \cdot\right\Vert _{Tr}$ denote the
usual trace norm induced by $Tr$ on $\mathcal{M}_{k}\left(  \mathbb{C}\right)
^{n}$, i.e.,%
\[
\left\Vert \left(  A_{1},\ldots,A_{n}\right)  \right\Vert _{Tr}=\sqrt
{Tr(A_{1}^{\ast}A_{1})+\ldots+Tr(A_{n}^{\ast}A_{n})}%
\]
for all $\left(  A_{1},\ldots,A_{n}\right)  $ in $\mathcal{M}_{k}\left(
\mathbb{C}\right)  ^{n}.$ Let $\left\Vert \cdot\right\Vert _{2}$ denote the
trace norm induced by $\tau_{k}$ on $\mathcal{M}_{k}\left(  \mathbb{C}\right)
^{n},$ i.e.,%
\[
\left\Vert \left(  A_{1},\ldots,A_{n}\right)  \right\Vert _{2}=\sqrt{\tau
_{k}(A_{1}^{\ast}A_{1})+\ldots+\tau_{k}(A_{n}^{\ast}A_{n})}%
\]
for all $\left(  A_{1},\ldots,A_{n}\right)  $ in $\mathcal{M}_{k}\left(
\mathbb{C}\right)  ^{n}.$

For every $\omega>0,$ we define the $\omega$-$\left\Vert \cdot\right\Vert
$-ball $Ball(B_{1},\ldots,B_{n};\omega,\left\Vert \cdot\right\Vert )$ centered
at $(B_{1},\ldots,B_{n})$ in $\mathcal{M}_{k}\left(  \mathbb{C}\right)  ^{n}$
to be the subset of $\mathcal{M}_{k}\left(  \mathbb{C}\right)  ^{n}$
consisting of all $(A_{1},\ldots,A_{n})$ in $\mathcal{M}_{k}\left(
\mathbb{C}\right)  ^{n}$ such that
\[
\left\Vert \left(  A_{1},\ldots,A_{n}\right)  -(B_{1},\ldots,B_{n})\right\Vert
<\omega.
\]

\begin{definition}
Suppose that $\sum$ is a subset of $\mathcal{M}_{k}\left(  \mathbb{C}\right)
^{n}.$We define $\nu_{\mathcal{1}}(\sum,$ $\omega)$ to be the minimal number
of $\omega$-$\left\Vert \cdot\right\Vert $-balls that cover $\sum$ in
$\mathcal{M}_{k}\left(  \mathbb{C}\right)  ^{n}.$
\end{definition}

For every $\omega>0,$ we define the $\omega$-$\left\Vert \cdot\right\Vert
_{2}$-ball $Ball$ $(B_{1},\ldots,B_{n};\omega,\left\Vert \cdot\right\Vert
_{2})$ centered at $(B_{1},\ldots,B_{n})$ in $\mathcal{M}_{k}\left(
\mathbb{C}\right)  ^{n}$ to be the subset of $\mathcal{M}_{k}\left(
\mathbb{C}\right)  ^{n}$ consisting of all $(A_{1},\ldots,A_{n})$ in
$\mathcal{M}_{k}\left(  \mathbb{C}\right)  ^{n}$ such that
\[
\left\Vert (A_{1},\ldots,A_{n})-\left(  B_{1},\ldots,B_{n}\right)  \right\Vert
_{2}<\omega.
\]

\begin{definition}
Suppose that $\sum$ is a subset of $\mathcal{M}_{k}\left(  \mathbb{C}\right)
^{n}.$ We define $\nu_{2}(\sum,$ $\omega)$ to be the minimal number of
$\omega$-$\left\Vert \cdot\right\Vert _{2}$-balls that cover $\sum$ in
$\mathcal{M}_{k}\left(  \mathbb{C}\right)  ^{n}.$
\end{definition}

There is a very deep result of S. Szarek \cite{SS} (see \cite{DH} for an
account of this result) concerning these covering numbers.

\begin{proposition}
\cite{SS} \label{Szarek} For each positive integer $n$ there is a $C_{n}>0$
such that, for every $k\in\mathbb{N}$, and every bounded set $E\subseteq
\mathcal{M}_{k}(\mathbb{C})^{n},$ and every $\omega>0$, we have
\[
\left(  \frac{1}{C_{n}}\right)  ^{k^{2}}\leq\frac{\nu_{\infty}\left(
E,\omega\right)  }{\nu_{2}\left(  E,\omega\right)  }\leq C_{n}^{k^{2}}.
\]

\end{proposition}

\subsection{Unitary Orbits of Balls in $\mathcal{M}_{k}(\mathbb{C})^{n}$.}

For every $\omega>0,$ we define the $\omega$-$orbit$-$\left\Vert
\cdot\right\Vert $-ball $\mathcal{U}(B_{1},\ldots,B_{n};\omega,\left\Vert
\cdot\right\Vert )$ centered at $(B_{1},\ldots,B_{n})$ in $\mathcal{M}%
_{k}(\mathbb{C})^{n}$ to be the subset of $\mathcal{M}_{k}(\mathbb{C})^{n}$
consisting of all $(A_{1},\ldots,A_{n})$ in $\mathcal{M}_{k}(\mathbb{C})^{n}$
such that there exists some unitary matrix $W$ in $\mathcal{U}_{k}$
satisfying
\[
\left\Vert (A_{1},\ldots,A_{n})-(WB_{1}W^{\ast},\ldots,WB_{n}W^{\ast
})\right\Vert <\omega.
\]

\begin{definition}
Suppose that $\sum$ is a subset of $\mathcal{M}_{k}(\mathbb{C})^{n}.$We define
$o_{\mathcal{1}}(\sum,\omega)$ to be the minimal number of $\omega$%
-$orbit$-$\left\Vert \cdot\right\Vert $-balls that cover $\sum$ in
$\mathcal{M}_{k}(\mathbb{C})^{n}.$
\end{definition}

For every $\omega>0,$ we define the $\omega$-$orbit$-$\left\Vert
\cdot\right\Vert _{2}$-ball $\mathcal{U}(B_{1},\ldots B_{n};\omega,\left\Vert
\cdot\right\Vert _{2})$ centered at $(B_{1},\ldots,B_{n})$ in $\mathcal{M}%
_{k}(\mathbb{C})^{n}$ to be the subset of $\mathcal{M}_{k}(\mathbb{C})^{n}$
consisting of all $(A_{1},\ldots,A_{n})$ in $\mathcal{M}_{k}(\mathbb{C})^{n}$
such that there exists some unitary matrix $W$ in $\mathcal{U}_{k}$ satisfying%
\[
\left\Vert (A_{1},\ldots,A_{n})-(WB_{1}W^{\ast},\ldots,WB_{n}W^{\ast
})\right\Vert _{2}<\omega.
\]

\begin{definition}
Suppose that $\sum$ is a subset of $\mathcal{M}_{k}(\mathbb{C})^{n}.$ We
define $o_{2}(\sum,\omega)$ to be the minimal number of $\omega$%
-$orbit$-$\left\Vert \cdot\right\Vert _{2}$-balls that cover $\sum$ in
$\mathcal{M}_{k}(\mathbb{C})^{n}.$
\end{definition}

\subsection{Noncommutative Polynomials.}

Let $\mathbb{C<}X_{1},\ldots,X_{n}>$ be the unital noncommutative polynomials
in the indeterminants $X_{1},\ldots,X_{n}.$ Let $\left\{  P_{r}\right\}
_{r=1}^{\mathcal{1}}$ be the collection of all noncommutative polynomials in
$\mathbb{C<}X_{1},\ldots,X_{n}>$ with rational-complex coefficients, i.e.,
coefficients in $\mathbb{Q}+i\mathbb{Q}$.

\subsection{Moments and Voiculescu's Microstates Space.}

Suppose $\mathcal{M}$ is a von Neumann algebra with a faithful normal tracial
state $\tau$ and $x_{1},\ldots,x_{n}$ are selfadjoint elements of
$\mathcal{M}$ such that $\mathcal{M}=W^{\ast}\left(  x_{1},\ldots
,x_{n}\right)  $. Suppose $m\left(  t_{1},\ldots,t_{n}\right)  $ is a monomial
in free variables $t_{1},\ldots,t_{n}$. The $m^{th}$ moment of $\vec
{x}=\left(  x_{1},\ldots,x_{n}\right)  $ is defined as%
\[
\tau\left(  m\left(  x_{1},\ldots,x_{n}\right)  \right)  .
\]
Suppose $\mathcal{N}$ is a von Neumann algebra with a faithful normal trace
$\rho$ and $\mathcal{N}$ is generated by $y_{1},\ldots,y_{n}$. The following
fact shows how the moments contain all the information about the algebras.

\begin{proposition}
\label{moments}The tuples $\vec{x}$ and $\vec{y}$ have the same moments, i.e.,
for every monomial $m$,%
\[
\tau\left(  m\left(  x_{1},\ldots,x_{n}\right)  \right)  =\rho\left(  m\left(
y_{1},\ldots,y_{n}\right)  \right)  ,
\]
if and only if there is a normal $\ast$-isomorphism $\pi:\mathcal{M}%
\rightarrow\mathcal{N}$ such that

\begin{enumerate}
\item $\pi\left(  x_{j}\right)  =y_{j}$ for $1\leq j\leq n,$ and

\item $\tau=\rho\circ\pi$.
\end{enumerate}
\end{proposition}

If $N$ is a positive integer and $\varepsilon>0$, we say that $\left(  \vec
{x},\tau\right)  $ and $\left(  \vec{y},\rho\right)  $ are $\left(
N,\varepsilon\right)  $\emph{-close} if
\[
\left\vert \tau\left(  m\left(  x_{1},\ldots,x_{n}\right)  \right)
-\rho\left(  m\left(  y_{1},\ldots,y_{n}\right)  \right)  \right\vert
<\varepsilon
\]
for all monomials $m$ with degree less than or equal to $N$.

If $N$ and $k$ are positive integers and $\varepsilon>0$, we define%
\[
\Gamma_{R}\left(  x_{1},\ldots,x_{n};k,\varepsilon,N\right)
\]
to be the set of all selfadjoint tuples $\vec{A}=\left(  A_{1},\ldots
,A_{n}\right)  $ of $k\times k$ complex matrices such that $\left\Vert \vec
{A}\right\Vert \leq R$ and $\left(  \vec{x},\tau\right)  $ and $\left(
\vec{A},\tau_{k}\right)  $ are $\left(  N,\varepsilon\right)  $-close.

\subsection{Voiculescu's Free Entropy Dimension.}

If $\mathcal{M}$ is a von Neumann algebra with a faithful normal trace $\tau$
and selfadjoint generators $x_{1},\ldots,x_{n}$, and $R\geq\left\Vert \left(
x_{1},\ldots,x_{n}\right)  \right\Vert $, we define the \emph{free entropy
dimension of }$\vec{x}$ by%
\[
\delta_{0}\left(  \vec{x}\right)  =\limsup_{\omega\rightarrow0^{+}}%
\inf_{N,\varepsilon}\limsup_{k\rightarrow\infty}\frac{\log\nu_{\infty}\left(
\Gamma_{R}\left(  \vec{x};k,\varepsilon,N\right)  \right)  }{-k^{2}\log\omega
},
\]
which turns out to be independent of $R$. It follows from Szarek's result,
Proposition \ref{Szarek}, that the definition of $\delta_{0}$ remains
unchanged if we replace $\nu_{\infty}$ with $\nu_{2}$.

If we are dealing with more than one trace, we use the notation%
\[
\delta_{0}\left(  x_{1},\ldots,x_{n};\tau\right)  .
\]

\subsection{Voiculescu's Norm-microstates Space.}

Suppose $\mathcal{A}$ is a unital C*-algebra generated by selfadjoint elements
$x_{1},\ldots,x_{n}$, and suppose $\left\{  P_{1},P_{2},\ldots\right\}  $ are
the polynomials in $n$ free variables with rational-complex coefficients. We
replace the moments in the von Neumann algebra setting with norms of
polynomials,
\[
\left\Vert P_{j}\left(  x_{1},\ldots,x_{n}\right)  \right\Vert .
\]
It is easy to see that the analogue of Proposition \ref{moments} holds.

\begin{proposition}
Suppose $\mathcal{A}=C^{\ast}\left(  x_{1},\ldots,x_{n}\right)  $ and
$\mathcal{B}=C^{\ast}\left(  y_{1},\ldots,y_{n}\right)  $. Then%
\[
\left\Vert P_{j}\left(  \vec{x}\right)  \right\Vert =\left\Vert P_{j}\left(
\vec{y}\right)  \right\Vert
\]
for $1\leq j<\infty$ if and only if there is a unital $\ast$-isomorphism
$\pi:\mathcal{A}\rightarrow\mathcal{B}$ such that $\pi\left(  x_{k}\right)
=y_{k}$ for $1\leq k\leq n$.
\end{proposition}

We say that $\vec{x}$ and $\vec{y}$ are \emph{topologically }$\left(
N,\varepsilon\right)  $\emph{-close} if\emph{ }%
\[
\left\vert \left\Vert P_{j}\left(  \vec{x}\right)  \right\Vert -\left\Vert
P_{j}\left(  \vec{y}\right)  \right\Vert \right\vert <\varepsilon
\]
for $1\leq j\leq N$.

\bigskip

For all integers $r,k\geq1,$ and $\varepsilon>0,$ and noncommutative
polynomials $P_{1},\ldots,P_{r},$ we define
\[
\Gamma^{\text{top}}(x_{1},\ldots,x_{n};k,\varepsilon,P_{1},\ldots,P_{r})
\]
to be the subset of $\left(  \mathcal{M}_{k}^{sa}(\mathbb{C)}\right)  ^{n}$
consisting of all the
\[
\left(  A_{1},\ldots,A_{n}\right)  \in\left(  \mathcal{M}_{k}^{sa}%
(\mathbb{C)}\right)  ^{n}%
\]
that are topologically $\left(  r,\varepsilon\right)  $-close to $\vec{x},$
i.e., satisfying%
\[
\left\vert \left\Vert P_{j}(A_{1},\ldots,A_{n})\right\Vert -\left\Vert
P_{j}(x_{1},\ldots,x_{n})\right\Vert \right\vert <\varepsilon,\mathcal{8}1\leq
j\leq r.
\]

\subsection{Voiculescu's Topological Free Entropy Dimension.}

Define
\[
\nu_{\mathcal{1}}(\Gamma^{\text{top}}(x_{1},\ldots,x_{n};k,\varepsilon
,P_{1},\ldots,P_{r}),\omega)
\]
to be the covering number of the set $\Gamma^{\text{top}}(x_{1},\ldots
,x_{n};k,\varepsilon,P_{1},\ldots,P_{r})$ by $\omega$-$\left\Vert
\cdot\right\Vert $-balls in the metric space $(\mathcal{M}_{k}^{sa}%
(\mathbb{C))}^{n}$ equipped with operator norm.

\begin{definition}
\textbf{The topological free entropy dimension }of $x_{1},\ldots,x_{n}$ is
defined by
\[
\delta_{top}(x_{1},\ldots,x_{n})=
\]%
\[
\underset{\omega\rightarrow0^{+}}{\lim\sup}\underset{\varepsilon
>0,r\in\mathbb{N}}{\inf}\underset{k\rightarrow\mathcal{1}}{\lim\sup}\frac
{\log(\nu_{\mathcal{1}}(\Gamma^{\text{top}}(x_{1},\ldots,x_{n};k,\varepsilon
,P_{1},\ldots,P_{r}),\omega))}{-k^{2}\log\omega}.
\]

\end{definition}

For each positive integer $N,$ define $\mathbb{P}_{N}\left(  t_{1}%
,\ldots,t_{n}\right)  $ to be the set of all $p$ in $\mathbb{C<}X_{1}%
,\ldots,X_{n}>$ of degree at most $N$ and whose coefficients have modulus at
most $N$. If $k\in\mathbb{N},$ and $\varepsilon>0$, we define
\[
\Gamma^{\text{top}}\left(  x_{1},\ldots,x_{n};N,\varepsilon,k\right)
\]
to be the set of all $\left(  A_{1},\ldots,A_{n}\right)  \in\mathcal{M}%
_{k}^{n}\left(  \mathbb{C}\right)  $ such that $\left\vert \left\Vert p\left(
\vec{x}\right)  \right\Vert -\left\Vert p\left(  \vec{A}\right)  \right\Vert
\right\vert <\varepsilon$ for every $\ast$-polynomial $p\in\mathbb{P}%
_{N}\left(  t_{1},\ldots,t_{n}\right)  $. It was shown in \cite{HS2} that%
\[
\delta_{top}(x_{1},\ldots,x_{n})=
\]%
\[
\underset{\omega\rightarrow0^{+}}{\lim\sup}\underset{\varepsilon
>0,r\in\mathbb{N}}{\inf}\underset{k\rightarrow\mathcal{1}}{\lim\sup}\frac
{\log(\nu_{\mathcal{1}}(\Gamma^{\text{top}}(x_{1},\ldots,x_{n};N,\varepsilon
,k),\omega))}{-k^{2}\log\omega}.
\]

It follows from a result of S. Szarek \cite{SS} (see \cite{DH} for an
exposition) that the above definitions of $\delta_{top}$ remains unchanged if
we replace $\nu_{\infty}$ with $\nu_{2}$.

\subsection{MF-algebras}

We note that the definition of $\delta_{top}(x_{1},\ldots,x_{n})$ makes sense
if and only if, for every $\varepsilon>0$ and every $r,k_{0}\in\mathbb{N}$,
there is a $k\geq k_{0}$ such that $\Gamma^{\text{top}}(x_{1},\ldots
,x_{n};k,\varepsilon,P_{1},\ldots,P_{r})\neq\varnothing$. In \cite{HLS} we
proved that this is equivalent to C*$\left(  x_{1},\ldots,x_{n}\right)  $
being an $MF$ C*-algebra in the sense of Blackadar and Kirchberg \cite{BK}. A
C*-algebra $\mathcal{A}$ is an \emph{MF-algebra} if $\mathcal{A}$ can be
embedded into $%
%TCIMACRO{\dprod _{1\leq k<\infty}}%
%BeginExpansion
{\displaystyle\prod_{1\leq k<\infty}}
%EndExpansion
\mathcal{M}_{m_{k}}\left(  \mathbb{C}\right)  /%
%TCIMACRO{\dsum _{1\leq k<\infty}}%
%BeginExpansion
{\displaystyle\sum_{1\leq k<\infty}}
%EndExpansion
\mathcal{M}_{m_{k}}\left(  \mathbb{C}\right)  $ for some increasing sequence
$\left\{  m_{k}\right\}  $ of positive integers. In particular C$^{\ast
}\left(  x_{1},\ldots,x_{n}\right)  $ is an $MF$-algebra if there is a
sequence $\left\{  m_{k}\right\}  $ of positive integers and sequences
$\left\{  A_{1k}\right\}  ,\ldots,\left\{  A_{nk}\right\}  $ with
$A_{1k},\ldots,A_{nk}\in\mathcal{M}_{m_{k}}\left(  \mathbb{C}\right)  $ such
that
\[
\lim_{k\rightarrow\infty}\left\Vert p\left(  A_{1k},\ldots,A_{nk}\right)
\right\Vert =\left\Vert p\left(  x_{1},\ldots,x_{n}\right)  \right\Vert
\]
for every $\ast$-polynomial $p\left(  t_{1},\ldots,t_{n}\right)  $.

When the above holds for every $\ast$-polynomial $p$, we say that the sequence
$\left\{  \vec{A}_{k}=\left(  A_{1k},\ldots,A_{nk}\right)  \right\}  $
converges to $\vec{x}=\left(  x_{1},\ldots,x_{n}\right)  $ in
\emph{topological distribution}, and write%
\[
\vec{A}_{k}\overset{t.d.}{\longrightarrow}\vec{x}.
\]

\subsection{ Noncommutative Continuous Functions}

The algebra of \emph{noncommutative continuous functions} of $n$ variables was
introduced and studied in \cite{HKM}. Basically, it is the metric completion
of the algebra of $\ast$-polynomials with respect to a family of seminorms.
There is a functional calculus for these functions on any $n$-tuple of
elements in any unital C*-algebra. Here is a list of the basic properties of
these functions \cite{HKM}:

\begin{enumerate}
\item For each such function $\varphi$ there is a sequence $\left\{
p_{n}\right\}  $ of noncommutative $\ast$-polynomials such that for every
tuple $\left(  T_{1},\ldots,T_{n}\right)  $ we have
\[
\left\Vert p_{n}\left(  T_{1},\ldots,T_{n}\right)  -\varphi\left(
T_{1},\ldots,T_{n}\right)  \right\Vert \rightarrow0,
\]
and the convergence is uniform on bounded $n$-tuples.

\item For any tuple $\left(  T_{1},\ldots,T_{n}\right)  ,$ $C^{\ast}\left(
T_{1},\ldots,T_{n}\right)  $ is the set of all $\varphi\left(  T_{1}%
,\ldots,T_{n}\right)  $ with $\varphi$ a noncommutative continuous function.

\item For any $n$-tuple $\left(  A_{1},\ldots,A_{n}\right)  $ and any $S\in
C^{\ast}\left(  A_{1},\ldots,A_{n}\right)  ,$ there is a noncommutative
continuous function $\varphi$ such that $S=\varphi\left(  A_{1},\ldots
,A_{n}\right)  $ and $\left\Vert \varphi\left(  T_{1},\ldots,T_{n}\right)
\right\Vert \leq\left\Vert S\right\Vert $ for all $n$-tuples $\left(
T_{1},\ldots,T_{n}\right)  .$

\item If $T_{1},\ldots,T_{n}$ are elements of a unital C*-algebra
$\mathcal{A}$ and $\pi:\mathcal{A}\rightarrow\mathcal{B}$ is a unital $\ast
$-homomorphism, then
\[
\pi\left(  \varphi\left(  T_{1},\ldots,T_{n}\right)  \right)  =\varphi\left(
\pi\left(  T_{1}\right)  ,\ldots,\pi\left(  T_{n}\right)  \right)
\]
for every noncommutative continuous function $\varphi$.
\end{enumerate}

It is clear that if $\vec{A}_{k}\overset{t.d.}{\longrightarrow}\vec{x}$, then
$\lim_{k\rightarrow\infty}\left\Vert \varphi\left(  A_{1k},\ldots
,A_{nk}\right)  \right\Vert =\left\Vert \varphi\left(  x_{1},\ldots
,x_{n}\right)  \right\Vert $ for every noncommutative continuous function
$\varphi$.

\subsection{Change of Variables}

The following change of variable theorem was proved in \cite{HS2}.

\begin{theorem}
\label{CV}Suppose $x_{1},\ldots,x_{n},y_{1},\ldots,y_{m}$ are elements of a
unital C*-algebra $\mathcal{A}$ and there are noncommutative continuous
functions $\varphi_{1},\ldots,\varphi_{m}$ in $n$ variables. Suppose also that

\begin{enumerate}
\item $y_{j}=\varphi_{j}\left(  x_{1},\ldots,x_{n}\right)  $ for $1\leq j\leq
m,$

\item $\left\Vert \left(  \varphi_{1}\left(  \vec{a}\right)  ,\ldots
,\varphi_{m}\left(  \vec{a}\right)  \right)  -(\varphi_{1}(\vec{b}%
),\ldots,\varphi_{m}(\vec{b}))\right\Vert \leq M\left\Vert \vec{a}-\vec
{b}\right\Vert $ for some fixed $M>0$ and all operator $n$-tuples $\vec
{a},\vec{b}$ with norm less than $1+\left\Vert \left(  x_{1},\ldots
,x_{n}\right)  \right\Vert .$

\item $x_{1},\ldots,x_{n}\in C^{\ast}\left(  y_{1},\ldots,y_{m}\right)  .$
\end{enumerate}

Then
\[
\delta_{\text{top}}\left(  x_{1},\ldots,x_{n}\right)  \geq\delta_{\text{top}%
}\left(  y_{1},\ldots,y_{m}\right)  .
\]

\end{theorem}

\subsection{ $\delta_{\text{top}}^{1/2}$}

In this section we discuss D. Voiculescu's notion of semi-microstates
\newline$\Gamma_{1/2}^{\text{top}}\left(  x_{1},\ldots,x_{n};p_{1}%
,\ldots,p_{r},\varepsilon,k\right)  $ and the corresponding invariant
$\delta_{\text{top}}^{1/2}\left(  x_{1},\ldots,x_{n}\right)  $. It turns out
that the domain of definition of $\delta_{\text{top}}^{1/2}\left(
x_{1},\ldots,x_{n}\right)  $ is much larger than that of $\delta_{\text{top}%
}\left(  x_{1},\ldots,x_{n}\right)  ,$ but that when $C^{\ast}\left(
x_{1},\ldots,x_{n}\right)  $ is an MF-algebra, they are equal.

\begin{definition}
$\Gamma_{1/2}^{\text{top}}(x_{1},\ldots,x_{n};k,\varepsilon,P_{1},\ldots
,P_{r})$ is the set of all $\left(  a_{1},\ldots,a_{n}\right)  \in
\mathcal{M}_{k}^{n}\left(  \mathbb{C}\right)  $ such that%
\[
\left\Vert P_{j}\left(  \vec{a}\right)  \right\Vert \leq\left\Vert
P_{j}\left(  \vec{x}\right)  \right\Vert +\varepsilon
\]
for $1\leq j\leq r$. Similarly, we define $\Gamma_{1/2}^{\text{top}}\left(
x_{1},\ldots,x_{n};N,\varepsilon,k\right)  $ for a positive integer $N$.

We define $\delta_{\text{top}}^{1/2}\left(  x_{1},\ldots,x_{n}\right)  $ to
be
\[
\limsup\limits_{\omega\rightarrow0^{+}}\inf\limits_{r\in\mathbb{N}%
,\varepsilon>0}\limsup\limits_{k\rightarrow\infty}\frac{\log\left(
\nu_{\mathcal{1}}\left(  \Gamma_{1/2}^{\text{top}}(x_{1},\ldots,x_{n}%
;k,\varepsilon,P_{1},\ldots,P_{r}),\omega\right)  \right)  }{-k^{2}\log\omega
}.
\]

\end{definition}

The following result was pointed out by Voiculescu \cite{DV5}.

\begin{theorem}
$\delta_{\text{top}}^{1/2}\left(  x_{1},\ldots,x_{n}\right)  =\delta
_{\text{top}}\left(  x_{1},\ldots,x_{n}\right)  $ whenever $\delta
_{\text{top}}\left(  x_{1},\ldots,x_{n}\right)  $ is defined.
\end{theorem}

The following result from \cite{HS3} simplifies some the lower bound estimates
in \cite{HS2}.

\begin{corollary}
\label{pig}%
\[
\delta_{\text{top}}^{1/2}\left(  x_{1},\ldots,x_{n}\right)  \geq\sup\left\{
\delta_{\text{top}}^{1/2}\left(  \pi\left(  x_{1}\right)  ,\ldots,\pi\left(
x_{n}\right)  \right)  :\pi\in\text{Rep}\left(  C^{\ast}\left(  x_{1}%
,\ldots,x_{n}\right)  \right)  \right\}  .
\]

\end{corollary}

\section{MF-traces\bigskip}

\subsection{Basic Properties\bigskip}

\begin{definition}
Suppose $\mathcal{A}=C^{\ast}\left(  x_{1},\ldots,x_{n}\right)  $ is an MF
C*-algebra. A tracial state $\tau$ on $\mathcal{A}$ is an $MF$\emph{-trace} if
there is a sequence $\left\{  m_{k}\right\}  $ of positive integers and
sequences $\left\{  A_{1k}\right\}  ,\ldots,\left\{  A_{nk}\right\}  $ with
$A_{1k},\ldots,A_{nk}\in\mathcal{M}_{m_{k}}\left(  \mathbb{C}\right)  $ such
that, for every $\ast$-polynomial $p$,

\begin{enumerate}
\item $\lim_{k\rightarrow\infty}\left\Vert p\left(  A_{1k},\ldots
,A_{nk}\right)  \right\Vert =\left\Vert p\left(  x_{1},\ldots,x_{n}\right)
\right\Vert $, and

\item $\lim_{k\rightarrow\infty}\tau_{m_{k}}\left(  p\left(  A_{1k}%
,\ldots,A_{nk}\right)  \right)  =$ $\tau\left(  p\left(  x_{1},\ldots
,x_{n}\right)  \right)  $.
\end{enumerate}

Recall that if $\left(  1\right)  $ above holds for every $\ast$-polynomial
$p$, we say that the sequence $\left\{  \vec{A}_{k}=\left(  A_{1k}%
,\ldots,A_{nk}\right)  \right\}  $ converges to $\vec{x}=\left(  x_{1}%
,\ldots,x_{n}\right)  $ in \emph{topological distribution}, and write%
\[
\vec{A}_{k}\overset{t.d.}{\longrightarrow}\vec{x},
\]
and when $\left(  2\right)  $ above holds, we say that $\left\{  \left(
\vec{A}_{k},\tau_{m_{k}}\right)  \right\}  $ converges to $\left(  \vec
{x},\tau\right)  $ \emph{in distribution}, and write%
\[
\left(  \vec{A}_{k},\tau_{m_{k}}\right)  \overset{\mathrm{dist}}%
{\longrightarrow}\left(  \vec{x},\tau\right)  .
\]
We let $\mathcal{TS}\left(  \mathcal{A}\right)  $ denote the set of all
tracial states on $\mathcal{A}$ and $\mathcal{T}_{MF}\left(  \mathcal{A}%
\right)  $ denote the set of all MF-traces on $\mathcal{A}$.
\end{definition}

\begin{remark}
It is easily seen that if $\vec{A}_{k}\overset{t.d.}{\longrightarrow}\vec{x}$,
then, for every noncommutative continuous function $\varphi\left(
t_{1},\ldots,t_{n}\right)  ,$ we have
\[
\lim_{k\rightarrow\infty}\left\Vert \varphi\left(  A_{1k},\ldots
,A_{nk}\right)  \right\Vert =\left\Vert \varphi\left(  x_{1},\ldots
,x_{n}\right)  \right\Vert .
\]
Indeed, if $\varepsilon>0$, then, by \cite{HKM}, there is a polynomial $p$
such that
\[
\left\Vert p\left(  \vec{A}\right)  -\varphi\left(  \vec{A}\right)
\right\Vert <\varepsilon/3
\]
for every $\vec{A}$ with $\left\Vert \vec{A}\right\Vert \leq\sup
_{k\in\mathbb{N}}\left\Vert \vec{A}_{k}\right\Vert $. It follows that%
\[
\left\vert \left\Vert \varphi\left(  \vec{A}_{k}\right)  \right\Vert
-\left\Vert \varphi\left(  \vec{x}\right)  \right\Vert \right\vert \leq
\]%
\[
\left\vert \left\Vert \varphi\left(  \vec{A}_{k}\right)  \right\Vert
-\left\Vert p\left(  \vec{A}_{k}\right)  \right\Vert \right\vert +\left\vert
\left\Vert p\left(  \vec{A}_{k}\right)  \right\Vert -\left\Vert p\left(
\vec{x}\right)  \right\Vert \right\vert +\left\vert \left\Vert p\left(
\vec{x}\right)  \right\Vert -\left\Vert \varphi\left(  \vec{x}\right)
\right\Vert \right\vert <
\]%
\[
2\varepsilon/3+\left\vert \left\Vert p\left(  \vec{A}_{k}\right)  \right\Vert
-\left\Vert p\left(  \vec{x}\right)  \right\Vert \right\vert ,
\]
which is clearly less than $\varepsilon$ when $k$ is sufficiently large.
\end{remark}

\bigskip

The following lemma is obvious.\bigskip

\begin{lemma}
\label{trace} Suppose $\mathcal{A}=C^{\ast}\left(  x_{1},\ldots,x_{n}\right)
$ is a unital $MF$-algebra and $\tau$ is a tracial state on $\mathcal{A}$.
Then $\tau\in\mathcal{T}_{MF}\left(  \mathcal{A}\right)  $ if and only if, for
every $\varepsilon>0$ and every finite set $\mathcal{F}$ of $\ast
$-polynomials, there is a positive integer $k$ and $A_{1},\ldots,A_{n}%
\in\mathcal{M}_{k}\left(  \mathbb{C}\right)  $ such that, for every
$p\in\mathcal{F}$,

\begin{enumerate}
\item $\left\vert \left\Vert p\left(  A_{1},\ldots,A_{n}\right)  \right\Vert
-\left\Vert p\left(  x_{1},\ldots,x_{n}\right)  \right\Vert \right\vert
<\varepsilon$, and

\item $\left\vert \tau_{k}\left(  p\left(  A_{1},\ldots,A_{n}\right)  \right)
-\tau\left(  p\left(  x_{1},\ldots,x_{n}\right)  \right)  \right\vert
<\varepsilon$.
\end{enumerate}
\end{lemma}

\bigskip

We say a tracial state $\tau$ on a unital C*-algebra $\mathcal{A}$ is
\emph{finite-dimensional} if there is a finite-dimensional C*-algebra
$\mathcal{B}$ with a tracial state $\rho$ and a unital $\ast$-homomorphism
$\pi:\mathcal{A}\rightarrow\mathcal{B}$ such that $\tau=\rho\circ\pi$. Then
there are positive integers $s_{1},\ldots,s_{w}$, nonnegative numbers
$t_{1},\ldots,t_{w}$ with $%
%TCIMACRO{\dsum _{j=1}^{w}}%
%BeginExpansion
{\displaystyle\sum_{j=1}^{w}}
%EndExpansion
t_{j}=1$, and unital $\ast$-homomorphisms $\pi_{j}:\mathcal{A}\rightarrow
\mathcal{M}_{s_{j}}\left(  \mathbb{C}\right)  ,$ for $1\leq j\leq w$, such
that, for every $a\in\mathcal{A}$, we have%
\[
\tau\left(  a\right)  =%
%TCIMACRO{\dsum _{j=1}^{w}}%
%BeginExpansion
{\displaystyle\sum_{j=1}^{w}}
%EndExpansion
t_{j}\tau_{s_{j}}\left(  \pi_{j}\left(  a\right)  \right)  .
\]
\bigskip

\bigskip

\begin{proposition}
\label{mft}Suppose $\mathcal{A}=C^{\ast}\left(  x_{1},\ldots,x_{n}\right)  $
is an $MF$-algebra. Then

\begin{enumerate}
\item $\mathcal{T}_{MF}\left(  \mathcal{A}\right)  $ is a nonempty
weak*-compact convex set.

\item Every finite-dimensional tracial state on $\mathcal{A}$ is in
$\mathcal{T}_{MF}\left(  \mathcal{A}\right)  $.

\item If $\pi$ is a unital $\ast$-homomorphism on $\mathcal{A}$ and
$\pi\left(  \mathcal{A}\right)  $ is an $MF$-algebra, then
\[
\left\{  \varphi\circ\pi:\varphi\in\mathcal{T}_{MF}\left(  \pi\left(
\mathcal{A}\right)  \right)  \right\}  \subseteq\mathcal{T}_{MF}\left(
\mathcal{A}\right)  .
\]

\item A tracial state $\psi$ on $\mathcal{A}$ is in $\mathcal{T}_{MF}\left(
\mathcal{A}\right)  $ if and only if there is a free ultrafilter $\alpha$ on
$\mathbb{N}$, and a unital $\ast$-homomorphism $\pi:\mathcal{A}\rightarrow%
%TCIMACRO{\dprod \limits^{\alpha}}%
%BeginExpansion
{\displaystyle\prod\limits^{\alpha}}
%EndExpansion
\mathcal{M}_{k}\left(  \mathbb{C}\right)  $ such that $\psi=\tau_{\alpha}%
\circ\pi,$ where
\[
\tau_{\alpha}\left(  \left\{  A_{k}\right\}  _{\alpha}\right)  =\lim
_{k\rightarrow\alpha}\tau_{k}\left(  A_{k}\right)  .
\]

\item If $\mathcal{B}$ is a unital C*-subalgebra of $\mathcal{A}$ and
$\varphi\in\mathcal{T}_{MF}\left(  \mathcal{A}\right)  ,$ then $\varphi
|_{\mathcal{B}}\in\mathcal{T}_{MF}\left(  \mathcal{B}\right)  $.

\item If $\mathcal{B}=C^{\ast}\left(  y_{1},\ldots,y_{m}\right)  $ is MF,
$\nu$ is a C*-tensor norm such that $\mathcal{A}\otimes_{\nu}\mathcal{B}$ is
$MF,$ and one of $\mathcal{A}$ and $\mathcal{B}$ is exact, $\alpha
\in\mathcal{T}_{MF}\left(  \mathcal{A}\right)  ,$ $\beta\in\mathcal{T}%
_{MF}\left(  \mathcal{B}\right)  $, then $\alpha\otimes\beta\in\mathcal{T}%
_{MF}\left(  \mathcal{A}\otimes_{\nu}\mathcal{B}\right)  $.
\end{enumerate}
\end{proposition}

\begin{proof}
$\left(  1\right)  $. It follows from the preceding lemma that $\mathcal{T}%
_{MF}\left(  \mathcal{A}\right)  $ is weak*-closed. For convexity it suffices
to show that $\left(  \tau+\rho\right)  /2\in\mathcal{T}_{MF}\left(
\mathcal{A}\right)  $ whenever $\tau,\rho\in\mathcal{T}_{MF}\left(
\mathcal{A}\right)  $. Choose sequences $\left\{  m_{k}\right\}  $ and
$\left\{  s_{k}\right\}  $ of positive integers and $\vec{A}_{k}\in
\mathcal{M}_{m_{k}}^{n}\left(  \mathbb{C}\right)  $ and $\vec{B}_{k}%
\in\mathcal{M}_{s_{k}}^{n}\left(  \mathbb{C}\right)  $ such that $\vec{A}%
_{k}\overset{t.d.}{\longrightarrow}\vec{x}$ and $\vec{B}_{k}\overset
{t.d.}{\longrightarrow}\vec{x}$ and such that, for every $\ast$-polynomial
$p$, we have
\[
\lim_{k\rightarrow\infty}\tau_{m_{k}}\left(  p\left(  \vec{A}_{k}\right)
\right)  =\tau\left(  p\left(  \vec{x}\right)  \right)  \text{ and }%
\lim_{k\rightarrow\infty}\tau_{s_{k}}\left(  p\left(  \vec{B}_{k}\right)
\right)  =\rho\left(  p\left(  \vec{x}\right)  \right)  .
\]
For each $k\in\mathbb{N}$, let $\vec{T}_{k}=\vec{A}_{k}^{\left(  s_{k}\right)
}\oplus\vec{B}_{k}^{\left(  m_{k}\right)  }\in\mathcal{M}_{2s_{k}m_{k}}%
^{n}\left(  \mathbb{C}\right)  $, where $D^{\left(  t\right)  }$ denotes a
direct sum of $t$ copies of the operator $D.$ It follows that $\vec{T}%
_{k}\overset{t.d.}{\longrightarrow}\vec{x}$ and, for every $\ast$ -polynomial
$p$, we have
\[
\lim_{k\rightarrow\infty}\tau_{2m_{k}s_{k}}\left(  p\left(  \vec{T}%
_{k}\right)  \right)  =\lim_{k\rightarrow\infty}\frac{Tr\left(  p\left(
\vec{A}_{k}^{\left(  s_{k}\right)  }\oplus\vec{B}_{k}^{\left(  m_{k}\right)
}\right)  \right)  }{2m_{k}s_{k}}%
\]%
\[
=\lim_{k\rightarrow\infty}\frac{s_{k}m_{k}\tau_{m_{k}}\left(  p\left(  \vec
{A}_{k}\right)  \right)  +m_{k}s_{k}\tau_{s_{k}}\left(  p\left(  \vec{B}%
_{k}\right)  \right)  }{2m_{k}s_{k}}=\left(  \tau\left(  \vec{x}\right)
+\rho\left(  \vec{x}\right)  \right)  /2.
\]
Hence $\mathcal{T}_{MF}\left(  \mathcal{A}\right)  $ is convex.

$\left(  2\right)  .$ Suppose $\tau$ is a finite-dimensional trace on
$\mathcal{A}$. Then there are positive integers $s_{1},\ldots,s_{w}$,
nonnegative numbers $t_{1},\ldots,t_{w}$ with $%
%TCIMACRO{\dsum _{j=1}^{w}}%
%BeginExpansion
{\displaystyle\sum_{j=1}^{w}}
%EndExpansion
t_{j}=1$, and unital $\ast$-homomorphisms $\pi_{j}:\mathcal{A}\rightarrow
\mathcal{M}_{s_{j}}\left(  \mathbb{C}\right)  ,$ for $1\leq j\leq w$, such
that, for every $a\in\mathcal{A}$, we have%
\[
\tau\left(  a\right)  =%
%TCIMACRO{\dsum _{j=1}^{w}}%
%BeginExpansion
{\displaystyle\sum_{j=1}^{w}}
%EndExpansion
t_{j}\tau_{s_{j}}\left(  \pi_{j}\left(  a\right)  \right)  .
\]
Since $\mathcal{T}_{MF}\left(  \mathcal{A}\right)  $ is weak*-compact, it is
sufficient to consider the case where each $t_{j}$ is a rational number, i.e.,
$t_{j}\acute{=}\frac{u_{j}}{v_{j}}$ with $u_{j},v_{j}\in\mathbb{N}$. Thus
there is a positive integer $N$ and a representation $\pi:\mathcal{A}%
\rightarrow\mathcal{M}_{N}\left(  \mathbb{C}\right)  $ such that, for every
$a\in\mathcal{A}$, $\tau\left(  a\right)  =\tau_{N}\left(  \pi\left(
a\right)  \right)  $. Since $\mathcal{A}$ is $MF$, there is a sequence
$\left\{  m_{k}\right\}  $ of positive integers and sequences $\left\{
A_{1k}\right\}  ,\ldots,\left\{  A_{nk}\right\}  $ with $A_{1k},\ldots
,A_{nk}\in\mathcal{M}_{m_{k}}\left(  \mathbb{C}\right)  $ such that, for every
$\ast$-polynomial $p$,
\[
\lim_{k\rightarrow\infty}\left\Vert p\left(  A_{1k},\ldots,A_{nk}\right)
\right\Vert =\left\Vert p\left(  x_{1},\ldots,x_{n}\right)  \right\Vert .
\]
For each $k\in\mathbb{N},$ and each $j,$ $1\leq j\leq n,$ define
\[
B_{jk}=A_{jk}\oplus\pi\left(  x_{j}\right)  ^{km_{k}}\in\mathcal{M}_{s_{k}%
}\left(  \mathbb{C}\right)  ,
\]
where $s_{k}=m_{k}\left(  1+kN\right)  $. It is clear that $\vec{B}%
_{k}=\left(  B_{1k},\ldots,B_{nk}\right)  \overset{t.d.}{\longrightarrow}%
\vec{x}$ and that $\lim_{k\rightarrow\infty}\tau_{s_{k}}\left(  p\left(
\vec{B}_{k}\right)  \right)  =\tau\left(  p\left(  \vec{x}\right)  \right)  $
for every $\ast$-polynomial $p$. Hence $\tau\in\mathcal{T}_{MF}\left(
\mathcal{A}\right)  $.

(3). Suppose $\varphi\in\mathcal{T}_{MF}\left(  \pi\left(  \mathcal{A}\right)
\right)  $ and choose a sequence $\vec{B}_{k}\in\mathcal{M}_{m_{k}}^{n}\left(
\mathbb{C}\right)  $ such that $\vec{B}_{k}\overset{t.d.}{\longrightarrow
}\left(  \pi\left(  x_{1}\right)  ,\ldots,\pi\left(  x_{n}\right)  \right)
=\pi\left(  \vec{x}\right)  $ and such that $\left(  \vec{B}_{k}%
,\varphi\right)  \overset{\text{\textrm{dist}}}{\longrightarrow}\left(
\pi\left(  \vec{x}\right)  ,\varphi\right)  .$ Since $\mathcal{A}$ is MF,
there is a sequence $\vec{A}_{k}\in\mathcal{M}_{s_{k}}^{n}\left(
\mathbb{C}\right)  $ such that $\vec{A}_{k}\overset{t.d.}{\longrightarrow}%
\vec{x}$. If we let $\vec{C}_{k}=\vec{A}_{k}\oplus\left(  \vec{B}_{k}\right)
^{\left(  ks_{k}\right)  }$, it is clear that $\vec{C}_{k}\overset
{t.d.}{\longrightarrow}\vec{x}$ and $\left(  \vec{C}_{k},\tau_{\left(
km_{k}+1\right)  s_{k}}\right)  \overset{\text{\textrm{dist}}}{\longrightarrow
}\left(  \vec{x},\varphi\circ\pi\right)  $.

(4). This is an immediate consequence of statement (3) and Lemma \ref{trace}.

(5). This follows from (4).

(6). We know from \cite[Proposition 3.2]{HS4} that $\mathcal{A}\otimes
_{\text{\textrm{min}}}\mathcal{B}$ is MF. Also there is a natural surjective
$\ast$-homomorphism $\pi:\mathcal{A}\otimes_{\nu}\mathcal{B}\rightarrow
\mathcal{A}\otimes_{\text{\textrm{min}}}\mathcal{B}$. Since $\alpha
\otimes\beta$ factors through $\pi$, it follows from part $\left(  3\right)  $
of this proposition that we need only show that $\alpha\otimes\beta
\in\mathcal{T}_{MF}\left(  \mathcal{A}\otimes_{\text{\textrm{min}}}%
\mathcal{B}\right)  $. However, the proof of \cite[Proposition 3.2]{HS4}
easily yields this fact.\bigskip\ 
\end{proof}

\subsection{The MF Ideal}

We know \cite{HLS} that if $\mathcal{A}$ $=\mathcal{K}\left(  \ell^{2}\right)
+\mathbb{C}1$, then
\[
\delta_{top}\left(  x_{1},\ldots,x_{n}\right)  =0
\]
for any generating set $\left\{  x_{1},\ldots,x_{n}\right\}  $. This is
because any trace on $\mathcal{A}$ must vanish on $\mathcal{K}\left(  \ell
^{2}\right)  $. We want to investigate this phenomenon further. Suppose
$\mathcal{A}$ is an $MF$-algebra. We define the \emph{MF-ideal} of
$\mathcal{\ A}$ as%
\[
\mathcal{J}_{MF}=\mathcal{J}_{MF}\left(  \mathcal{A}\right)  =\left\{
a\in\mathcal{A}:\forall\tau\in\mathcal{T}_{MF}\left(  \mathcal{A}\right)
,\text{ }\tau\left(  a^{\ast}a\right)  =0\right\}  .
\]
\bigskip

It is easy to describe the elements of $\mathcal{J}_{MF}\left(  \mathcal{A}%
\right)  $ in terms of noncommutative continuous functions. Recall from
\cite{HKM} that
\[
C^{\ast}\left(  x_{1},\ldots,x_{n}\right)  =\left\{  \varphi\left(
x_{1},\ldots,x_{n}\right)  :\varphi\text{ is a noncommutative continuous
function}\right\}  .
\]

\begin{lemma}
\label{ncmf} Suppose $\mathcal{A}=C^{\ast}\left(  x_{1},\ldots,x_{n}\right)  $
and $\varphi$ is a noncommutative continuous function of $n$ variables. The
following are equivalent:

\begin{enumerate}
\item $\varphi\left(  x_{1},\ldots,x_{n}\right)  \in\mathcal{J}_{MF}\left(
\mathcal{A}\right)  .$

\item Whenever $\vec{A}_{k}\in\mathcal{M}_{m_{k}}^{n}\left(  \mathbb{C}%
\right)  $ and $\vec{A}_{k}\overset{t.d.}{\longrightarrow}\vec{x}$, we have%
\[
\lim_{k\rightarrow\infty}\left\Vert \varphi\left(  \vec{A}_{k}\right)
\right\Vert _{2}=\lim_{k\rightarrow\infty}\left[  \tau_{m_{k}}\left(
\varphi\left(  \vec{A}_{k}\right)  ^{\ast}\varphi\left(  \vec{A}_{k}\right)
\right)  \right]  ^{1/2}=0.
\]

\item For every $\omega>0$, there is an $\varepsilon_{0}>0$, $N_{0},k_{0}%
\in\mathbb{N}$, such that, for every $0<\varepsilon<\varepsilon_{0},k\geq
k_{0}$, and $N\geq N_{0},$ and for every $\vec{A}\in\Gamma^{\text{\textrm{top}%
}}\left(  \vec{x};N,\varepsilon,k\right)  $, we have%
\[
\left\Vert \varphi\left(  \vec{A}\right)  \right\Vert _{2}<\omega.
\]

\end{enumerate}
\end{lemma}

\begin{proof}
The equivalence of statements $\left(  2\right)  $ and $\left(  3\right)  $ is
obvious, as is the equivalence of statements $\left(  1\right)  $ and $\left(
2\right)  $.
\end{proof}

\bigskip

Since, with respect to the $\left\Vert \cdot\right\Vert _{2}$-norm in the
topological $\Gamma$-sets, the elements corresponding to the elements of
$\mathcal{J}_{MF}\left(  \mathcal{A}\right)  $ converge to $0,$ it might seem
that $\delta_{top}$ may only depend on $\mathcal{A}/\mathcal{J}_{MF}\left(
\mathcal{A}\right)  $. However, one possible complication is that
$\mathcal{A}/\mathcal{J}_{MF}\left(  \mathcal{A}\right)  $ might not be an
$MF$-algebra. Here are two sample positive results.

For the next theorem we need to set up some notation. Suppose $k$ is a (large)
positive integer and $d_{1}\leq\cdots\leq d_{s}$ are positive integers.
Suppose $m_{1},\ldots,m_{s}$ are nonnegative integers such that%
\[%
%TCIMACRO{\dsum _{t=1}^{s}}%
%BeginExpansion
{\displaystyle\sum_{t=1}^{s}}
%EndExpansion
d_{t}m_{t}\leq k.
\]
We define a (not necessarily unital) $\ast$-homomorphism
\[
\rho:\mathcal{M}_{d_{1}}\left(  \mathbb{C}\right)  \oplus\cdots\oplus
\mathcal{M}_{d_{s}}\left(  \mathbb{C}\right)  \rightarrow\mathcal{M}%
_{k}\left(  \mathbb{C}\right)
\]
by
\[
\rho\left(  A_{1}\oplus\cdots\oplus A_{s}\right)  =A_{1}^{\left(
m_{1}\right)  }\oplus\cdots\oplus A_{s}^{\left(  m_{s}\right)  }\oplus0,
\]
where $A_{1}^{\left(  m_{1}\right)  }\oplus\cdots\oplus A_{s}^{\left(
m_{s}\right)  }\oplus0$ is the block diagonal $k\times k$ matrix whose blocks
are $m_{1}$ copies of $A_{1}$, followed by $m_{2}$ copies of $A_{2},\ldots,$
followed by $m_{s}$ copies of $A_{s}$ with the remaining block (if any)
consisting of a zero matrix. We call such a representation $\rho$ a
\textbf{canonical representation} of $\mathcal{M}_{d_{1}}\left(
\mathbb{C}\right)  \oplus\cdots\oplus\mathcal{M}_{d_{s}}\left(  \mathbb{C}%
\right)  $ in $\mathcal{M}_{k}\left(  \mathbb{C}\right)  .$ Note that the
canonical representation $\rho$ is completely determined by the choice of
$m_{1},\ldots,m_{s}$, and so the number of canonical representations is no
more than $k^{s}$. In \cite{HS3} it was shown that if $\dim C^{\ast}\left(
x_{1},\ldots,x_{n}\right)  =d<\infty$, then $\delta_{top}\left(  x_{1}%
,\ldots,x_{n}\right)  =1-\frac{1}{d}.$

\begin{theorem}
Suppose $\mathcal{A}=C^{\ast}\left(  x_{1},\ldots,x_{n}\right)  $ is an
$MF$-algebra and $\mathcal{A}/\mathcal{J}_{MF}\left(  \mathcal{A}\right)  $
has dimension $d<\infty$. Then%
\[
\delta_{top}\left(  x_{1},\ldots,x_{n}\right)  =1-\frac{1}{d}.
\]

\end{theorem}

\begin{proof}
It follows from the change of variables theorem that we can assume that
$\left\Vert x_{1}\right\Vert ,\ldots,\left\Vert x_{n}\right\Vert \leq1$ and we
can add $1$ to the generating set, so we can assume that $x_{1}=1$. Since
$\mathcal{A}/\mathcal{J}_{MF}\left(  \mathcal{A}\right)  $ is
finite-dimensional, there is a surjective unital $\ast$-homomorphism
$\pi:\mathcal{A}\rightarrow\mathcal{M}_{d_{1}}\left(  \mathbb{C}\right)
\oplus\cdots\oplus\mathcal{M}_{d_{s}}\left(  \mathbb{C}\right)  $ with
$\ker\pi=\mathcal{J}_{MF}\left(  \mathcal{A}\right)  $. It follows from
Corollary \ref{pig} and \cite{HS3} that
\[
\delta_{top}\left(  x_{1},\ldots,x_{n}\right)  \geq\delta_{top}\left(
\pi\left(  x_{1}\right)  ,\ldots,\pi\left(  x_{n}\right)  \right)  =1-\frac
{1}{d}.
\]
Suppose $\omega>0$.

\textbf{Claim:} There is an $\varepsilon_{0}>0$ and $k_{0},N_{0}\in\mathbb{N}$
such that, for every $0<\varepsilon<\varepsilon_{0},$ every $N\geq N_{0},$
every $k\geq k_{0},$ and every $\vec{A}\in\Gamma^{\text{top}}(x_{1}%
,\ldots,x_{n};k,\varepsilon,P_{1},\ldots,P_{N})$ there is a canonical
representation $\rho:\mathcal{M}_{d_{1}}\left(  \mathbb{C}\right)
\oplus\cdots\oplus\mathcal{M}_{d_{s}}\left(  \mathbb{C}\right)  \rightarrow
\mathcal{M}_{k}\left(  \mathbb{C}\right)  $ and a unitary matrix
$U\in\mathcal{M}_{k}\left(  \mathbb{C}\right)  $ such that%
\[%
%TCIMACRO{\dsum _{r=1}^{n}}%
%BeginExpansion
{\displaystyle\sum_{r=1}^{n}}
%EndExpansion
\left\Vert A_{r}-U\rho\left(  \pi\left(  x_{r}\right)  \right)  U^{\ast
}\right\Vert _{2}<\omega/4,
\]
and%
\[
1-\tau_{k}\left(  \rho\left(  1\right)  \right)  <\omega/4.
\]

\textbf{Proof of Claim:} Assume the claim is false. Then, for every positive
integer $m$, there is a positive integer $k_{m}\geq m$ and an $\vec{A}%
_{m}=\left(  A_{m1},\ldots,A_{mn}\right)  \in\Gamma^{\text{top}}\left(
x_{1},\ldots,x_{n};m,\frac{1}{m},k_{m}\right)  $ such that, for every
canonical representation
\[
\rho:\mathcal{M}_{d_{1}}\left(  \mathbb{C}\right)  \oplus\cdots\oplus
\mathcal{M}_{d_{s}}\left(  \mathbb{C}\right)  \rightarrow\mathcal{M}_{k_{m}%
}\left(  \mathbb{C}\right)
\]
and every unitary matrix $U\in\mathcal{M}_{k_{m}}\left(  \mathbb{C}\right)
$,
\[%
%TCIMACRO{\dsum _{r=1}^{n}}%
%BeginExpansion
{\displaystyle\sum_{r=1}^{n}}
%EndExpansion
\left\Vert A_{mr}-U\rho\left(  \pi\left(  x_{r}\right)  \right)  U^{\ast
}\right\Vert _{2}\geq\omega/4.
\]
Note that any subsequence of $\left\{  \vec{A}_{m}\right\}  $ has the same
properties, so we can assume that there is a $\tau\in\mathcal{T}_{MF}\left(
\mathcal{A}\right)  $ such that%
\[
\left(  \vec{A}_{m},\tau_{k_{m}}\right)  \overset{\mathrm{dist}}%
{\longrightarrow}\left(  \vec{x},\tau\right)  .
\]
We know from the definition of $\left\{  \vec{A}_{m}\right\}  $ that%
\[
\vec{A}_{m}\overset{\mathrm{t.d.}}{\longrightarrow}\vec{x}.
\]
We now let $\alpha$ be a free ultrafilter on $\mathbb{N}$, and we let $\left(
\mathcal{N},\sigma\right)  $ be the tracial ultraproduct $%
%TCIMACRO{\dprod \limits^{\alpha}}%
%BeginExpansion
{\displaystyle\prod\limits^{\alpha}}
%EndExpansion
\left(  \mathcal{M}_{k_{m}}\left(  \mathbb{C}\right)  ,\tau_{k_{m}}\right)  .$
Let $y_{j}=\left\{  A_{mj}\right\}  _{\alpha}\in\mathcal{N}$ for $1\leq j\leq
n$. It follows that, for every noncommutative polynomial $p$,
\[
\left\Vert p\left(  \vec{y}\right)  \right\Vert \leq\lim_{m\rightarrow\alpha
}\left\Vert p\left(  \vec{A}_{m}\right)  \right\Vert =\left\Vert p\left(
\vec{x}\right)  \right\Vert .
\]
Hence $\pi_{0}:\mathcal{A}\rightarrow\mathcal{N}$ defined by $\pi_{0}\left(
p\left(  \vec{x}\right)  \right)  =p\left(  \vec{y}\right)  $ is a unital
$\ast$-homomorphism. Moreover, $\tau=\sigma\circ\pi_{0},$ and since $\sigma$
is faithful on $\mathcal{N}$, we have%
\[
\ker\pi=\mathcal{J}_{MF}\left(  \mathcal{A}\right)  \subseteq\ker\pi_{0}.
\]
Hence there is a unital $\ast$-homomorphism $\pi_{1}:\mathcal{M}_{d_{1}%
}\left(  \mathbb{C}\right)  \oplus\cdots\oplus\mathcal{M}_{d_{s}}\left(
\mathbb{C}\right)  \rightarrow\mathcal{N}$ such that%
\[
\pi_{0}=\pi_{1}\circ\pi.
\]
Since $\sigma\circ\pi_{1}$ is a tracial state on $\mathcal{M}_{d_{1}}\left(
\mathbb{C}\right)  \oplus\cdots\oplus\mathcal{M}_{d_{s}}\left(  \mathbb{C}%
\right)  $, there are $t_{1},\ldots,t_{s}\geq0$ with $%
%TCIMACRO{\dsum _{j=1}^{s}}%
%BeginExpansion
{\displaystyle\sum_{j=1}^{s}}
%EndExpansion
t_{j}=1,$ such that%
\[
\sigma\circ\pi_{1}\left(  T_{1}\oplus\cdots\oplus T_{s}\right)  =%
%TCIMACRO{\dsum _{j=1}^{s}}%
%BeginExpansion
{\displaystyle\sum_{j=1}^{s}}
%EndExpansion
t_{j}\tau_{d_{j}}\left(  T_{j}\right)  .
\]
For each $\varepsilon>0$ we can find positive rational numbers of the form
$\frac{z_{1}\left(  \varepsilon\right)  }{d\left(  \varepsilon\right)
},\ldots,\frac{z_{s}\left(  \varepsilon\right)  }{d\left(  \varepsilon\right)
}$ such that
\[%
%TCIMACRO{\dsum _{r=1}^{s}}%
%BeginExpansion
{\displaystyle\sum_{r=1}^{s}}
%EndExpansion
\left\vert t_{r}-\frac{z_{r}\left(  \varepsilon\right)  }{d\left(
\varepsilon\right)  }\right\vert <\varepsilon,
\]
and%
\[%
%TCIMACRO{\dsum _{r=1}^{s}}%
%BeginExpansion
{\displaystyle\sum_{r=1}^{s}}
%EndExpansion
\frac{z_{r}\left(  \varepsilon\right)  }{d\left(  \varepsilon\right)  }=1.
\]
For each positive integer $m,$ let $u_{\varepsilon rm}$ be the largest integer
not greater than $k_{m}/d_{r}d\left(  \varepsilon\right)  $ and note that
\[
\lim_{m\rightarrow\infty}\frac{u_{\varepsilon rm}}{k_{m}}=\frac{1}%
{d_{r}d\left(  \varepsilon\right)  }.
\]
We define a canonical representation
\[
\rho_{\varepsilon m}:\mathcal{M}_{d_{1}}\left(  \mathbb{C}\right)
\oplus\cdots\oplus\mathcal{M}_{d_{s}}\left(  \mathbb{C}\right)  \rightarrow
\mathcal{M}_{k_{m}}\left(  \mathbb{C}\right)
\]
by%
\[
\rho_{\varepsilon m}\left(  T\right)  =T_{1}^{\left(  z_{1}\left(
\varepsilon\right)  u_{\varepsilon1m}\right)  }\oplus\cdots\oplus
T_{s}^{\left(  z_{s}\left(  \varepsilon\right)  u_{\varepsilon sm}\right)
}\oplus0.
\]
Define a representation $\rho_{\varepsilon}:\mathcal{M}_{d_{1}}\left(
\mathbb{C}\right)  \oplus\cdots\oplus\mathcal{M}_{d_{s}}\left(  \mathbb{C}%
\right)  \rightarrow\mathcal{N}$ by
\[
\rho_{\varepsilon}\left(  T\right)  =\left\{  \rho_{\varepsilon m}\left(
T\right)  \right\}  _{\alpha}.
\]
It is clear that%
\[
\left(  \sigma\circ\rho_{\varepsilon}\right)  \left(  T\right)  =\lim
_{m\rightarrow\alpha}\frac{1}{k_{m}}%
%TCIMACRO{\dsum _{r=1}^{s}}%
%BeginExpansion
{\displaystyle\sum_{r=1}^{s}}
%EndExpansion
z_{r}\left(  \varepsilon\right)  u_{\varepsilon rm}d_{r}\tau_{d_{r}}\left(
T_{r}\right)  =%
%TCIMACRO{\dsum _{r=1}^{s}}%
%BeginExpansion
{\displaystyle\sum_{r=1}^{s}}
%EndExpansion
\frac{z_{r}\left(  \varepsilon\right)  }{d\left(  \varepsilon\right)  }%
\tau_{d_{r}}\left(  T_{r}\right)  .
\]
It is clear that, as $\varepsilon\rightarrow0^{+}$, we have%
\[
\left(  \left(  \rho_{\varepsilon}\left(  \pi\left(  x_{1}\right)  \right)
,\ldots,\rho_{\varepsilon}\left(  \pi\left(  x_{n}\right)  \right)  \right)
,\sigma\right)  \rightarrow\left(  \left(  y_{1},\ldots,y_{n}\right)
,\sigma\right)  .
\]
It follows from \cite{DH1} that there is an $\varepsilon>0$ and a unitary
element $U\in\mathcal{N}$ such that%
\[%
%TCIMACRO{\dsum _{j=1}^{n}}%
%BeginExpansion
{\displaystyle\sum_{j=1}^{n}}
%EndExpansion%
%TCIMACRO{\dsum _{r=1}^{n}}%
%BeginExpansion
{\displaystyle\sum_{r=1}^{n}}
%EndExpansion
\left\Vert y_{r}-U\rho_{\varepsilon}\left(  \pi\left(  x_{r}\right)  \right)
U^{\ast}\right\Vert _{2}<\omega/4.
\]
We can write $U=\left\{  U_{m}\right\}  _{\alpha}$, where each $U_{m}$ is a
unitary matrix in $\mathcal{M}_{k_{m}}\left(  \mathbb{C}\right)  $. We then
have%
\[
\omega/4>%
%TCIMACRO{\dsum _{j=1}^{n}}%
%BeginExpansion
{\displaystyle\sum_{j=1}^{n}}
%EndExpansion%
%TCIMACRO{\dsum _{r=1}^{n}}%
%BeginExpansion
{\displaystyle\sum_{r=1}^{n}}
%EndExpansion
\left\Vert y_{r}-U\rho_{\varepsilon}\left(  \pi\left(  x_{r}\right)  \right)
U^{\ast}\right\Vert _{2}=\lim_{m\rightarrow\infty}%
%TCIMACRO{\dsum _{j=1}^{n}}%
%BeginExpansion
{\displaystyle\sum_{j=1}^{n}}
%EndExpansion%
%TCIMACRO{\dsum _{r=1}^{n}}%
%BeginExpansion
{\displaystyle\sum_{r=1}^{n}}
%EndExpansion
\left\Vert A_{mr}-U_{m}\rho_{m\varepsilon}\left(  \pi\left(  x_{r}\right)
\right)  U_{m}^{\ast}\right\Vert _{2},
\]
which implies there is an $m$ such that%
\[%
%TCIMACRO{\dsum _{j=1}^{n}}%
%BeginExpansion
{\displaystyle\sum_{j=1}^{n}}
%EndExpansion%
%TCIMACRO{\dsum _{r=1}^{n}}%
%BeginExpansion
{\displaystyle\sum_{r=1}^{n}}
%EndExpansion
\left\Vert A_{mr}-U_{m}\rho_{m\varepsilon}\left(  \pi\left(  x_{r}\right)
\right)  U_{m}^{\ast}\right\Vert _{2}<\omega/4.
\]
This contradiction implies that the claim is true.

Now suppose $0<\varepsilon<\varepsilon_{0},$ $k\geq k_{0},$ and $N\geq N_{0}$.
Let $\mathcal{S}_{k}$ denote the set of canonical representations $\rho$ of
$\mathcal{M}_{d_{1}}\left(  \mathbb{C}\right)  \oplus\cdots\oplus
\mathcal{M}_{d_{s}}\left(  \mathbb{C}\right)  $ in $\mathcal{M}_{k}\left(
\mathbb{C}\right)  $ with $1-\tau_{k}\left(  \rho\left(  1\right)  \right)
<\omega/4$. As mentioned in the sentence before this theorem, the cardinality
of $\mathcal{S}_{k}$ is at most $k^{s}$. Suppose $\rho\in\mathcal{S}_{k}$. Let
$r_{\rho}=rank\left(  1-\rho\left(  1\right)  \right)  $. We have%
\[
r_{\rho}=k\left(  \tau_{k}\left(  1-\rho\left(  1\right)  \right)  \right)
<k\omega/4.
\]
We know that the commutant $\mathcal{C}$ of the range of $\rho$ is unitarily
equivalent to the algebra $\mathcal{M}_{m_{1}}\left(  \mathbb{C}\right)
^{\left(  d_{1}\right)  }\oplus\cdots\oplus\mathcal{M}_{m_{s}}\left(
\mathbb{C}\right)  ^{d_{s}}\oplus\mathcal{M}_{r_{\rho}}\left(  \mathbb{C}%
\right)  .$ It follows from a result of S. Szarek \cite{SS} (see \cite{DH} for
an equally useful, but more elementary result) that there is a constant $C$
(independent of $k$) and a set $\mathcal{W}_{\rho}$ of unitary matrices with%
\[
Card\left(  \mathcal{W}_{\rho}\right)  \leq\left(  \frac{C}{\omega}\right)
^{k^{2}-\left(  m_{1}^{2}+\cdots+m_{s}^{2}+r_{\rho}^{2}\right)  }%
\]
such that, for every unitary matrix $U$ there is a $W\in\mathcal{W}_{\rho}$
and a unitary $V$ in $\mathcal{C}^{\prime}$ such that $\left\Vert
U-WV\right\Vert <\omega/4$. This inequality implies
\[
\left\Vert U\rho\left(  y_{i}\right)  U^{\ast}-W\rho\left(  y_{i}\right)
W^{\ast}\right\Vert _{2}=\left\Vert U\rho\left(  y_{i}\right)  U^{\ast}%
-WV\rho\left(  y_{i}\right)  V^{\ast}W^{\ast}\right\Vert _{2}<2\left(
\omega/4\right)  .
\]
Hence, for every $\vec{A}\in\Gamma^{\text{top}}(x_{1},\ldots,x_{n}%
;k,\varepsilon,P_{1},\ldots,P_{N})$ there is a $\rho\in\mathcal{S}_{k}$ and a
$W\in\mathcal{W}_{\rho}$ such that%
\[
\left\Vert \vec{A}-W\rho\left(  \vec{y}\right)  W^{\ast}\right\Vert
_{2}<\omega.
\]
Hence%
\[
\nu_{2}\left(  \Gamma^{\text{top}}(x_{1},\ldots,x_{n};k,\varepsilon
,P_{1},\ldots,P_{N}),\omega\right)  \leq k^{s}\left(  \frac{C}{\omega}\right)
^{k^{2}-\left(  m_{1}^{2}+\cdots+m_{s}^{2}+r_{\rho}^{2}\right)  },
\]
which implies%
\[
\frac{\log\nu_{2}\left(  \Gamma^{\text{top}}(x_{1},\ldots,x_{n};k,\varepsilon
,P_{1},\ldots,P_{N}),\omega\right)  }{k^{2}}\leq
\]%
\[
s\frac{\log k}{k^{2}}+\left[  1-%
%TCIMACRO{\dsum _{j=1}^{s}}%
%BeginExpansion
{\displaystyle\sum_{j=1}^{s}}
%EndExpansion
\left(  \frac{m_{j}}{k}\right)  ^{2}\right]  \left[  \log C-\log\omega\right]
.
\]
Since $%
%TCIMACRO{\dsum _{j=1}^{s}}%
%BeginExpansion
{\displaystyle\sum_{j=1}^{s}}
%EndExpansion
\frac{m_{j}}{k}d_{j}=\tau_{k}\left(  \rho\left(  1\right)  \right)
\geq1-\omega/4,$ it follows from the Cauchy-Schwartz inequality that%
\[%
%TCIMACRO{\dsum _{j=1}^{s}}%
%BeginExpansion
{\displaystyle\sum_{j=1}^{s}}
%EndExpansion
\left(  \frac{m_{j}}{k}\right)  ^{2}\geq\frac{%
%TCIMACRO{\dsum _{j=1}^{s}}%
%BeginExpansion
{\displaystyle\sum_{j=1}^{s}}
%EndExpansion
\frac{m_{j}}{k}d_{j}}{%
%TCIMACRO{\dsum _{j=1}^{s}}%
%BeginExpansion
{\displaystyle\sum_{j=1}^{s}}
%EndExpansion
d_{j}^{2}}\geq\frac{\left(  1-\omega/4\right)  ^{2}}{d}.
\]
It clearly follows from the definition of $\delta_{top}$ that
\[
\delta_{top}\left(  x_{1},\ldots,x_{n}\right)  \leq1-\frac{1}{d}.
\]

\end{proof}

A C*-algebra $\mathcal{A}$ is \emph{residually finite dimensional }%
(\emph{RFD}) if the finite-dimensional representations of $\mathcal{A}$
separate the points of $\mathcal{A}$. Every $RFD$ algebra is $MF$. Combining
the preceding result with Corollary \ref{pig}, we obtain the following
theorem. For the details of the simple proof see the proof in the next section
of a much stronger result (Corollary \ref{rfd2}).

\begin{theorem}
Suppose $\mathcal{A}$ is a nuclear MF C*-algebra and $\mathcal{A}%
/\mathcal{J}_{MF}\left(  \mathcal{A}\right)  $ is an $RFD$ algebra. Then, for
any generators $x_{1},\ldots,x_{n}$ of $\mathcal{A}$ , we have%
\[
\delta_{top}\left(  x_{1},\ldots,x_{n}\right)  =1-1/\dim\left(  \mathcal{A}%
/\mathcal{J}_{MF}\left(  \mathcal{A}\right)  \right)  .
\]

\end{theorem}

\bigskip

\section{MF-nuclear Algebras}

Recall that a C*-algebra is nuclear if $\pi\left(  \mathcal{A}\right)
^{\prime\prime}$ is hyperfinite for every representation $\pi:\mathcal{A}%
\rightarrow B\left(  M\right)  $ for some Hilbert space $M$. We will say that
$\mathcal{A}=C^{\ast}\left(  x_{1},\ldots,x_{n}\right)  $ is \emph{MF-nuclear}
if $\pi_{\tau}\left(  \mathcal{A}\right)  ^{\prime\prime}$ is hyperfinite for
every $\tau\in\mathcal{T}_{MF}\left(  \mathcal{A}\right)  $. Since every
MF-trace can be factored through $\mathcal{A}/\mathcal{J}_{MF}\left(
\mathcal{A}\right)  $, it follows that if $\mathcal{A}/\mathcal{J}_{MF}\left(
\mathcal{A}\right)  $ is nuclear, then $\mathcal{A}$ is MF-nuclear.

The following Theorem contains some properties of MF-nuclearity.

\bigskip

\begin{theorem}
Suppose $\mathcal{A}$ and $\mathcal{B}$ are C*-algebras. The following are true.

\begin{enumerate}
\item If $\mathcal{A}$ is MF-nuclear, $\mathcal{B}$ is MF and $\pi
:\mathcal{A}\rightarrow\mathcal{B}$ is a surjective $\ast$-homomorphism, then
$\mathcal{B}$ is MF-nuclear.

\item $\mathcal{A}\oplus\mathcal{B}$ is MF-nuclear if and only if
$\mathcal{A}$ and $\mathcal{B}$ are both MF-nuclear.

\item For each $n\in\mathbb{N}$, $\mathcal{A}$ is MF-nuclear if and only if
$\mathcal{M}_{n}\left(  \mathcal{A}\right)  =\mathcal{M}_{n}\left(
\mathbb{C}\right)  \otimes\mathcal{A}$ is MF-nuclear.

\item If $\mathcal{A}$ and $\mathcal{B}$ are MF-nuclear and $\mathcal{A}%
\otimes_{\gamma}\mathcal{B}$ is MF for some C*-crossnorm $\gamma$, then
$\mathcal{A}\otimes_{\gamma}\mathcal{B}$ is MF-nuclear.

\item If $\mathcal{A}$ and $\mathcal{B}$ are MF-nuclear and either
$\mathcal{A}$ or $\mathcal{B}$ is exact, then $\mathcal{A}\otimes_{\min
}\mathcal{B}$ is MF-nuclear.

\item A direct limit of MF-nuclear C*-algebras is MF-nuclear.
\end{enumerate}
\end{theorem}

\begin{proof}
Statement $\left(  1\right)  $ follows from part $\left(  4\right)  $ of
Proposition \ref{mft} and statements $\left(  2\right)  $ and $\left(
3\right)  $ are obvious.

$\left(  4\right)  $. Suppose $\varphi$ is an MF-trace for $\mathcal{A}%
\otimes_{\gamma}\mathcal{B}$. Then the restriction $\varphi|\mathcal{A}\otimes
I\in\mathcal{T}_{MF}\left(  \mathcal{A}\otimes I\right)  $ and $\varphi
|I\otimes\mathcal{B}\in\mathcal{T}_{MF}\left(  I\otimes\mathcal{B}\right)  $.
Let $\left(  \pi,H,e\right)  $ be the GNS representation for $\varphi$. Then
$\hat{\varphi}:\pi\left(  \mathcal{A}\otimes_{\gamma}\mathcal{B}\right)
^{\prime\prime}\rightarrow\mathbb{C}$ defined by $\hat{\varphi}\left(
T\right)  =\left(  Te,e\right)  $ is a faithful trace. Since $\left(
\pi|\mathcal{A}\otimes I,\left[  \pi\left(  \mathcal{A}\otimes I\right)
e\right]  ^{-},e\right)  $ is the GNS construction for $\varphi|\mathcal{A}%
\otimes I$, and since $\mathcal{A}\otimes I$ is MF-nuclear, $\pi\left(
\mathcal{A}\otimes I\right)  ^{\prime\prime}|\left[  \pi\left(  \mathcal{A}%
\otimes I\right)  e\right]  ^{-}$ is a hyperfinite von Neumann algebra. Since
$\hat{\varphi}$ is faithful, it follows that the map $T\mapsto T|\left[
\pi\left(  \mathcal{A}\otimes I\right)  e\right]  ^{-}$ is a normal
isomorphism on $\pi\left(  \mathcal{A}\otimes I\right)  ^{\prime\prime}$;
whence, $\pi\left(  \mathcal{A}\otimes I\right)  ^{\prime\prime}$ is
hyperfinite. Similarly, $\pi\left(  I\otimes\mathcal{B}\right)  ^{\prime
\prime}$ is hyperfinite. However, each of the algebras $\pi\left(
\mathcal{A}\otimes I\right)  ^{\prime\prime}$ and $\pi\left(  I\otimes
\mathcal{B}\right)  ^{\prime\prime}$ are contained in the commutant of the
other. Thus $\pi\left(  \mathcal{A}\otimes_{\beta}\mathcal{B}\right)
^{\prime\prime}=\left[  \pi\left(  \mathcal{A}\otimes I\right)  ^{\prime
\prime}\cup\pi\left(  I\otimes\mathcal{B}\right)  ^{\prime\prime}\right]
^{\prime\prime}$ is hyperfinite, since the C*-algebra generated by two
commuting finite-dimensional (or nuclear) C*-algebras is a homomorphic image
of their tensor products. Hence $\mathcal{A}\otimes_{\gamma}\mathcal{B}$ is MF-nuclear.

$\left(  5\right)  $. This follows from $\left(  4\right)  $ and the fact
\cite[Proposition 3.1]{HS4} that the minimal tensor product of two MF-algebras
is MF if one of them is exact.

$\left(  6\right)  $. Suppose $\left\{  \mathcal{A}_{n}\right\}
_{n\in\mathbb{N}}$ is an increasing sequence of MF-algebras and $\mathcal{A}%
=\left[  \cup_{n\in\mathbb{N}}\mathcal{A}_{n}\right]  ^{-}$. We know from
\cite{BK} that $\mathcal{A}$ is MF. Suppose $\varphi\in\mathcal{T}_{MF}\left(
\mathcal{A}\right)  $ with GNS construction $\left(  \pi,H,e\right)  $. It
follows from part $\left(  5\right)  $ of Proposition \ref{mft} that
$\varphi|\mathcal{A}_{n}\in\mathcal{T}_{MF}\left(  \mathcal{A}_{n}\right)  $
for each $n\in\mathbb{N}$. Arguing as in the proof of part $\left(  4\right)
$, we see that $\pi\left(  \mathcal{A}_{n}\right)  ^{\prime\prime}$ is
hyperfinite for each $n\in\mathbb{N}$. Thus $\pi\left(  \mathcal{A}\right)
^{\prime\prime}=\left[  \cup_{n\in\mathbb{N}}\pi\left(  \mathcal{A}%
_{n}\right)  ^{\prime\prime}\right]  ^{-WOT}$ is hyperfinite. Hence
$\mathcal{A}$ is MF-nuclear.
\end{proof}

\bigskip

\bigskip

The importance of hyperfiniteness is due to the fact that, in the presence of
hyperfiniteness, $\delta_{0}$ can be computed using covering numbers of
unitary orbits. If $\vec{A}=\left(  A_{1},\ldots,A_{n}\right)  \in
\mathcal{M}_{k}^{n}\left(  \mathbb{C}\right)  ,$ we define the \emph{unitary
orbit} of $\vec{A}$ by%
\[
\mathcal{U}\left(  \vec{A}\right)  =\left\{  \left(  UA_{1}U,UA_{2}U^{\ast
},\ldots,UA_{n}U^{\ast}\right)  :U\in\mathcal{U}_{k}\right\}  .
\]
The following result is from \cite{DH}.

\begin{theorem}
\cite{DH} \label{hyperfinite} If $W^{\ast}\left(  x_{1},\ldots,x_{n}\right)  $
is hyperfinite with trace $\tau,$ and if there are sequences $\left\{
k_{s}\right\}  $ in $\mathbb{N}$ and $\vec{A}_{s}$ in $\mathcal{M}_{k_{s}}%
^{n}\left(  \mathbb{C}\right)  $ such that $\left(  \vec{A}_{s},\tau_{k_{s}%
}\right)  \overset{\mathrm{dist}}{\longrightarrow}\left(  \vec{x},\tau\right)
$ and $\left\{  \left\Vert \vec{A}_{s}\right\Vert \right\}  $ is bounded,
then
\[
\delta_{0}\left(  \vec{x};\tau\right)  =\underset{\omega\rightarrow0^{+}}%
{\lim\sup}\underset{s\rightarrow\infty}{\lim\sup}\frac{\log{\small \nu
}_{\infty}\left(  \mathcal{U}\left(  \vec{A}_{s}\right)  {\tiny ,\omega
}\right)  }{{\tiny -k}_{s}^{2}\log{\tiny \omega}}.
\]

\end{theorem}

\bigskip

\begin{theorem}
If $\mathcal{A}=C^{\ast}\left(  x_{1},\ldots,x_{n}\right)  $ is MF, $\tau
\in\mathcal{T}_{MF}\left(  \mathcal{A}\right)  $, and $\pi_{\tau}\left(
\mathcal{A}\right)  ^{\prime\prime}$ is hyperfinite, then%
\[
\delta_{\text{\textrm{top}}}(x_{1},\ldots,x_{n})\geq\delta_{0}\left(
x_{1},\ldots,x_{n};\tau\right)  .
\]
\bigskip
\end{theorem}

\begin{proof}
Choose sequences $\left\{  k_{s}\right\}  $ in $\mathbb{N}$ and $\vec{A}_{s}$
in $\mathcal{M}_{k_{s}}^{n}\left(  \mathbb{C}\right)  $such that $\left(
\vec{A}_{s},\tau_{k_{s}}\right)  \overset{\mathrm{dist}}{\longrightarrow
}\left(  \vec{x},\tau\right)  $ and $\vec{A}_{s}\overset{t.d.}{\longrightarrow
}\vec{x}.$ Then, for $\varepsilon>0,$ $N\in\mathbb{N}$, there is an $s_{0}%
\in\mathbb{N}$ such that, for all $s\geq s_{0}$,\bigskip%
\[
\vec{A}_{s}\in\Gamma^{\text{top}}\left(  \vec{x};N,\varepsilon,k_{s}\right)
\text{.}%
\]
Hence, for every $s\geq s_{0}$,
\[
\mathcal{U}\left(  \vec{A}_{s}\right)  \subseteq\Gamma^{\text{top}}\left(
\vec{x};N,\varepsilon,k_{s}\right)  .
\]
It now follows from Theorem \ref{hyperfinite} and the definition of
$\delta_{\text{\textrm{top}}}(x_{1},\ldots,x_{n})$ that%
\[
\delta_{\text{\textrm{top}}}(x_{1},\ldots,x_{n})\geq\delta_{0}\left(
x_{1},\ldots,x_{n};\tau\right)  .
\]

\end{proof}

\bigskip

\begin{corollary}
If $\mathcal{A}=C^{\ast}\left(  x_{1},\ldots,x_{n}\right)  $ is MF-nuclear,
then
\[
\delta_{\text{\textrm{top}}}(x_{1},\ldots,x_{n})\geq\sup_{\tau\in
\mathcal{T}_{MF}\left(  \mathcal{A}\right)  }\delta_{0}\left(  x_{1}%
,\ldots,x_{n};\tau\right)  .
\]

\end{corollary}

\bigskip

In \cite{HLS} it was proved that if $\mathcal{A}=C^{\ast}\left(  x_{1}%
,\ldots,x_{n}\right)  $ is MF and nuclear, then $\delta_{\text{\textrm{top}}%
}(x_{1},\ldots,x_{n})\leq1$. We can actually prove this remains true when
$\mathcal{A}$\textbf{ }is MF-nuclear.

\begin{theorem}
\label{MFnuclear}If $\mathcal{A}=C^{\ast}\left(  x_{1},\ldots,x_{n}\right)  $
is MF-nuclear, then $\delta_{\text{\textrm{top}}}(x_{1},\ldots,x_{n})\leq1$.
\end{theorem}

\begin{proof}
It follows from the change of variables theorem that we can assume that
$\left\Vert x_{j}\right\Vert <1$ for $1\leq j\leq n$. Suppose $k,d\in
\mathbb{N}$ and $d\leq k$. We can canonically (non-unitally) embed
$\mathcal{M}_{d}\left(  \mathbb{C}\right)  $ into $\mathcal{M}_{k}\left(
\mathbb{C}\right)  $ with a block-diagonal map $\sigma_{d,k}:\mathcal{M}%
_{d}\left(  \mathbb{C}\right)  \rightarrow\mathcal{M}_{k}\left(
\mathbb{C}\right)  $ so that
\[
\sigma_{d,k}\left(  A\right)  =\left(
\begin{array}
[c]{ccccc}%
A &  &  &  & \\
& A &  &  & \\
&  & \ddots &  & \\
&  &  & A & \\
&  &  &  & 0
\end{array}
\right)  .
\]
where the size of the $0$ matrix is smaller than $d\times d$. If $d>k,$ we
define $\sigma_{d,k}\left(  A\right)  =0$ for every $A.$

Suppose $\vec{A}=\left(  A_{1},\ldots,A_{n}\right)  \in\mathcal{M}_{k}\left(
\mathbb{C}\right)  ^{n}.$ Suppose $\omega>0$. We define $\Delta\left(  \vec
{A},k,\omega\right)  $ to be the smallest $d\in\mathbb{N}$ for which there
exist a unitary $U\in\mathcal{M}_{k}\left(  \mathbb{C}\right)  $ such that,
for some contractions $B_{1},\ldots,B_{n}\in U^{\ast}\sigma_{d,k}\left(
\mathcal{M}_{d}\left(  \mathbb{C}\right)  \right)  U$, we have%
\[%
%TCIMACRO{\dsum _{j=1}^{n}}%
%BeginExpansion
{\displaystyle\sum_{j=1}^{n}}
%EndExpansion
\left\Vert A_{j}-B_{j}\right\Vert _{2}<\omega.
\]

\textbf{Claim:} There is an $\varepsilon_{0}>0,$ $k_{0},N_{0},D\in\mathbb{N}$
such that for every $0<\varepsilon<\varepsilon_{0},$ and every $k\geq k_{0}$
and $N\geq N_{0}$ and every $\vec{A}\in\Gamma^{\text{\textrm{top}}}\left(
x_{1},\ldots,x_{n};N,\varepsilon,k\right)  ,$ we have $\Delta\left(  \vec
{A},k,\omega\right)  \leq D.$

\textbf{Proof of Claim:} Assume via contradiction that the claim is false.
Then for each positive integer $m$ there is a $k_{m}\geq m$ and $\vec{A}%
_{m}=\left(  A_{m1},\ldots,A_{mn}\right)  \in\Gamma^{\text{top}}\left(
x_{1},\ldots,x_{n};m,\frac{1}{m},k_{m}\right)  $ such that $\Delta\left(
\vec{A}_{m},k_{m},\omega\right)  \geq m$. Note that any subsequence of
$\left\{  \vec{A}_{m}\right\}  $ has the same properties, so we can assume
that there is a $\tau\in\mathcal{T}_{MF}\left(  \mathcal{A}\right)  $ such
that%
\[
\left(  \vec{A}_{m},\tau_{k_{m}}\right)  \overset{\mathrm{dist}}%
{\longrightarrow}\left(  \vec{x},\tau\right)  .
\]
We know from the definition of $\left\{  \vec{A}_{m}\right\}  $ that%
\[
\vec{A}_{m}\overset{\mathrm{t.d.}}{\longrightarrow}\vec{x}.
\]
We now let $\alpha$ be a free ultrafilter on $\mathbb{N}$, and we let $\left(
\mathcal{N},\rho\right)  $ be the tracial ultraproduct $%
%TCIMACRO{\dprod \limits^{\alpha}}%
%BeginExpansion
{\displaystyle\prod\limits^{\alpha}}
%EndExpansion
\left(  \mathcal{M}_{k_{m}}\left(  \mathbb{C}\right)  ,\tau_{k_{m}}\right)  .$
Let $y_{j}=\left\{  A_{mj}\right\}  _{\alpha}\in\mathcal{N}$ for $1\leq j\leq
n$. It follows that, for every noncommutative polynomial $p$,
\[
\left\Vert p\left(  \vec{y}\right)  \right\Vert \leq\lim_{m\rightarrow\alpha
}\left\Vert p\left(  \vec{A}_{m}\right)  \right\Vert =\left\Vert p\left(
\vec{x}\right)  \right\Vert .
\]
Hence $\pi:\mathcal{A}\rightarrow\mathcal{N}$ defined by $\pi\left(  p\left(
\vec{x}\right)  \right)  =p\left(  \vec{y}\right)  $ is a unital $\ast
$-homomorphism. Moreover, we have
\[
\left\Vert \pi\left(  p\left(  \vec{x}\right)  \right)  \right\Vert _{2}%
^{2}=\left\Vert p\left(  \vec{y}\right)  \right\Vert _{2}^{2}=\lim
_{m\rightarrow\alpha}\left\Vert p\left(  \vec{A}_{m}\right)  \right\Vert
_{2}^{2}=\tau\left(  p\left(  \vec{x}\right)  ^{\ast}p\left(  \vec{x}\right)
\right)  .
\]
This gives a trace-preserving isomorphism between $\pi\left(  \mathcal{A}%
\right)  ^{\prime\prime}\subseteq\mathcal{N}$ and $\pi_{\tau}\left(
\mathcal{A}\right)  ^{\prime\prime}$ (where $\pi_{\tau}$ is the GNS
representation for $\tau$). Since $\mathcal{A}$ is MF-nuclear, $\pi_{\tau
}\left(  \mathcal{A}\right)  ^{\prime\prime}$ is hyperfinite. By \cite{S} (see
also \cite{HL}) $\mathcal{N}$ is a $II_{1}$ factor; thus, by \cite{KR}, there
is a $d\in\mathbb{N}$ and a unital $\ast$-homomorphism $\eta:\mathcal{M}%
_{d}\left(  \mathbb{C}\right)  \rightarrow\mathcal{N}$ and $w_{1},\ldots
,w_{n}\in\mathcal{M}_{d}\left(  \mathbb{C}\right)  $ such that
\[
\left\Vert w_{j}\right\Vert \leq\left\Vert y_{j}\right\Vert \leq\left\Vert
x_{j}\right\Vert <1\text{ for }1\leq j\leq n,\text{ and}%
\]%
\[%
%TCIMACRO{\dsum _{j=1}^{n}}%
%BeginExpansion
{\displaystyle\sum_{j=1}^{n}}
%EndExpansion
\left\Vert y_{j}-\eta\left(  w_{j}\right)  \right\Vert _{2}<\omega.
\]
Since $\mathcal{M}_{d}\left(  \mathbb{C}\right)  $ has a unique tracial state,
we know $\tau_{d}=\rho\circ\eta$, so $\left(  \eta\left(  \vec{w}\right)
,\rho\right)  \overset{\mathrm{dist}}{=}\left(  \vec{w},\tau_{d}\right)  .$ We
can write $\eta\left(  w_{j}\right)  =\left\{  W_{mj}\right\}  _{\alpha}%
\in\mathcal{N}$ with each $\left\Vert W_{mj}\right\Vert \leq\left\Vert
w_{j}\right\Vert \leq1$. We then have%
\[%
%TCIMACRO{\dsum _{j=1}^{n}}%
%BeginExpansion
{\displaystyle\sum_{j=1}^{n}}
%EndExpansion
\left\Vert y_{j}-\eta\left(  w_{j}\right)  \right\Vert _{2}=\lim
_{m\rightarrow\alpha}%
%TCIMACRO{\dsum _{j=1}^{n}}%
%BeginExpansion
{\displaystyle\sum_{j=1}^{n}}
%EndExpansion
\left\Vert A_{mj}-W_{mj}\right\Vert _{2}<\omega.
\]
We know that
\[
\left(  \vec{W}_{m},\tau_{k_{m}}\right)  \overset{\mathrm{dist}}%
{\longrightarrow}\left(  \eta\left(  \vec{w}\right)  ,\rho\right)
\overset{\mathrm{dist}}{=}\left(  \vec{w},\tau_{d}\right)  .
\]
But we also know
\[
\left(  \left(  \sigma_{d,k_{m}}\left(  w_{1}\right)  ,\ldots,\sigma_{d,k_{m}%
}\left(  w_{n}\right)  \right)  ,\tau_{k_{m}}\right)  \overset{\mathrm{dist}%
}{\longrightarrow}\left(  \vec{w},\tau_{d}\right)  \overset{\mathrm{dist}}%
{=}\left(  \eta\left(  \vec{w}\right)  ,\rho\right)  .
\]
It follows from \cite{DH1} that, for each $m$, there is a unitary $U_{m}%
\in\mathcal{U}_{k_{m}}$ such that
\[
\lim_{m\rightarrow\alpha}%
%TCIMACRO{\dsum _{j=1}^{n}}%
%BeginExpansion
{\displaystyle\sum_{j=1}^{n}}
%EndExpansion
\left\Vert W_{m_{j}}-U_{m}^{\ast}\sigma_{d,k_{m}}\left(  w_{j}\right)
U_{m}\right\Vert _{2}=0.
\]
Hence
\[
\lim_{m\rightarrow\alpha}%
%TCIMACRO{\dsum _{j=1}^{n}}%
%BeginExpansion
{\displaystyle\sum_{j=1}^{n}}
%EndExpansion
\left\Vert A_{m_{j}}-U_{m}^{\ast}\sigma_{d,k_{m}}\left(  w_{j}\right)
U_{m}\right\Vert _{2}<\omega,
\]
which implies, eventually along $\alpha,$ that%
\[
m\leq\Delta\left(  \vec{A}_{m},k_{m},\omega\right)  \leq d.
\]
This contradiction implies the claim is true.

It follows from the claim, for $0<\varepsilon<\varepsilon_{0},$ $k\geq k_{0},$
$N\geq N_{0}$, that any covering of%
\[
\bigcup_{d=1}^{D}\left\{  \left(  U^{\ast}B_{1}U,\ldots,U^{\ast}B_{n}U\right)
:B_{1},\ldots,B_{n}\in\mathcal{M}_{d}\left(  \mathbb{C}\right)  ,U\text{
unitary}\right\}
\]
with $\omega$-$\left\Vert \cdot\right\Vert _{2}$-balls also covers
$\Gamma^{\text{\textrm{top}}}\left(  x_{1},\ldots,x_{n};N,\varepsilon
,k\right)  .$ We can choose an $\frac{\omega}{3}$-net $\mathcal{B}$ in
ball$\mathcal{M}_{d}\left(  \mathbb{C}\right)  ^{d}$ with Card$\left(
\mathcal{B}\right)  \leq\left(  \frac{1}{3\omega}\right)  ^{2nd^{2}}%
\leq\left(  \frac{1}{3\omega}\right)  ^{2nD^{2}},$ and we can choose an
$\frac{\omega}{3}$-net $\mathcal{V}$ in the set of unitary $k\times k$
matrices with Card$\left(  \mathcal{V}\right)  \leq\left(  \frac{9\pi
e}{\omega}\right)  ^{k^{2}}.$ Hence,
\[
\nu_{2}\left(  \Gamma^{\text{\textrm{top}}}\left(  x_{1},\ldots,x_{n}%
;N,\varepsilon,k\right)  ,\omega\right)  \leq D\left(  \frac{1}{3\omega
}\right)  ^{2nD^{2}}\left(  \frac{9\pi e}{\omega}\right)  ^{k^{2}}.
\]
Using the definition of $\delta_{\text{\textrm{top}}}$, we see that
\[
\delta_{\text{\textrm{top}}}\left(  x_{1},\ldots,x_{n}\right)  \leq1.
\]

\end{proof}

\bigskip

The ideas in the proof of Theorem \ref{MFnuclear} easily yield the following result.

\begin{corollary}
Suppose $\mathcal{A}=C^{\ast}\left(  x_{1},\ldots,x_{n}\right)  $ is an
MF-algebra and
\[
\left\{  x_{1},\ldots,x_{n}\right\}  =\cup_{j=1}^{m}E_{j},
\]
where $C^{\ast}\left(  E_{j}\right)  $ is MF-nuclear for $1\leq j\leq m$. Then
$\delta_{\text{\textrm{top}}}\left(  x_{1},\ldots,x_{n}\right)  \leq m$.
\end{corollary}

\section{A General Lower Bound\bigskip}

In the von Neumann algebra version of free entropy dimension, we know (see
\cite{KJ2} and \cite{DH}) that
\[
\delta_{0}\left(  x_{1},\ldots,x_{n}\right)  \geq\sup\left\{  \delta
_{0}\left(  x\right)  :x=x^{\ast}\in W^{\ast}\left(  x_{1},\ldots
,x_{n}\right)  \right\}  \text{.}%
\]
The following result from \cite{HLS} shows that the analog of this result for
$\delta_{\text{\textrm{top}}}$ is not true.

\begin{proposition}
\label{compact}Suppose $\vec{T}=\left(  T_{1},\ldots,T_{m}\right)  $ is an
irreducible tuple of operators that are "scalar+compact" on a separable
infinite-dimensional Hilbert space $H$. Then $\delta_{\text{\textrm{top}}%
}\left(  \vec{T}\right)  =0$.
\end{proposition}

There is still a hybrid version that is true. Recall that if $\pi_{\tau
}:\mathcal{A}\rightarrow B\left(  H\right)  $ and $e$ make up the GNS
construction for a tracial state $\tau$ on $\mathcal{A}$, then $\hat{\tau}%
:\pi\left(  \mathcal{A}\right)  ^{\prime\prime}\rightarrow\mathbb{C}$ is the
faithful tracial state defined by $\hat{\tau}\left(  T\right)  =\left(
Te,e\right)  $.\bigskip

\begin{theorem}
\label{lower}Suppose $\mathcal{A}=C^{\ast}\left(  x_{1},\ldots,x_{n}\right)  $
is a unital $MF$-algebra and $\tau\in\mathcal{T}_{MF}\left(  \mathcal{A}%
\right)  $. Suppose $b=b^{\ast}\in\pi_{\tau}\left(  \mathcal{A}\right)
^{\prime\prime}$. Then%
\[
\delta_{top}\left(  x_{1},\ldots,x_{n}\right)  \geq\delta_{0}\left(
b,\hat{\tau}\right)  .
\]

\end{theorem}

\begin{proof}
First suppose $b=\pi_{\tau}\left(  f\left(  x_{1},\ldots,x_{x}\right)
\right)  $ for some $\ast$-polynomial $f=f^{\ast}$. There is an $M>0$ such
that for all operators $A_{1},B_{1}\ldots,A_{n},B_{n}$ with $\left\Vert
A_{j}\right\Vert ,\left\Vert B_{j}\right\Vert \leq\left\Vert x_{j}\right\Vert
+1$ for $1\leq j\leq n,$ we have%
\[
\left\Vert f\left(  A_{1},\ldots A_{n}\right)  -f\left(  B_{1},\ldots
,B_{n}\right)  \right\Vert \leq M\left\Vert \left(  A_{1},\ldots,A_{n}\right)
-\left(  B_{1},\ldots,B_{n}\right)  \right\Vert .
\]

Since $\tau\in\mathcal{T}_{MF}\left(  \mathcal{A}\right)  ,$ there is a
sequence $\left\{  m_{k}\right\}  $ of positive integers and sequences
$\left\{  A_{1k}\right\}  ,\ldots,\left\{  A_{nk}\right\}  $ with
$A_{1k},\ldots,A_{nk}\in\mathcal{M}_{m_{k}}\left(  \mathbb{C}\right)  $ such
that, for every $\ast$-polynomial $p$, $\lim_{k\rightarrow\infty}\left\Vert
p\left(  A_{1k},\ldots,A_{nk}\right)  \right\Vert =\left\Vert p\left(
x_{1},\ldots,x_{n}\right)  \right\Vert $, and
\[
\lim_{k\rightarrow\infty}\tau_{m_{k}}\left(  p\left(  A_{1k},\ldots
,A_{nk}\right)  \right)  =\tau\left(  p\left(  x_{1},\ldots,x_{n}\right)
\right)  .
\]
If $\varepsilon>0$ and $N\in\mathbb{N}$, then there is a $k_{0}\in\mathbb{N}$
such that for all $k\geq k_{0},$ we have $\vec{A}_{k}\in\Gamma^{\text{top}%
}\left(  \vec{x};N,\varepsilon,m_{k}\right)  $ and $\vec{A}_{k}\in
\Gamma\left(  \pi_{\tau}\left(  \vec{x}\right)  ;N,m_{k},\varepsilon;\hat
{\tau}\right)  $. It follows that $\left(  f\left(  \vec{A}_{k}\right)
,\tau_{m_{k}}\right)  \overset{dist}{\longrightarrow}\left(  f\left(
\pi_{\tau}\left(  \vec{x}\right)  \right)  ,\hat{\tau}\right)  =\left(
b,\hat{\tau}\right)  $. We now have%
\begin{align*}
\delta_{top}\left(  \vec{x}\right)   &  \geq\limsup_{\omega\rightarrow0^{+}%
}\inf_{\varepsilon,N}\limsup_{k\rightarrow\infty}\frac{\log\nu_{2}\left(
\Gamma^{\text{top}}\left(  \vec{x};N,\varepsilon,m_{k}\right)  ,\omega\right)
}{-m_{k}^{2}\log\omega}\\
\geq &  \limsup_{\omega\rightarrow0^{+}}\inf_{\varepsilon,N}\limsup
_{k\rightarrow\infty}\frac{\log\nu_{2}\left(  \mathcal{U}\left(  \vec{A}%
_{k}\right)  ,\omega\right)  }{-m_{k}^{2}\log\omega}.
\end{align*}

It follows from the definition of $M$ that, for any $\omega$-net $\mathcal{S}$
for $\mathcal{U}\left(  \vec{A}_{k}\right)  $, we have $f\left(
\mathcal{S}\right)  $ is an $M\omega$-net for $f\left(  \mathcal{U}\left(
\vec{A}_{k}\right)  \right)  =\mathcal{U}\left(  f\left(  \vec{A}_{k}\right)
\right)  $. Hence we have%
\[
\limsup_{\omega\rightarrow0^{+}}\inf_{\varepsilon,N}\limsup_{k\rightarrow
\infty}\frac{\log\nu_{2}\left(  \mathcal{U}\left(  \vec{A}_{k}\right)
,\omega\right)  }{-m_{k}^{2}\log\omega}%
\]%
\[
\geq\limsup_{\omega\rightarrow0^{+}}\inf_{\varepsilon,N}\limsup_{k\rightarrow
\infty}\frac{\log\nu_{2}\left(  \mathcal{U}\left(  f\left(  \vec{A}%
_{k}\right)  \right)  ,M\omega\right)  }{-m_{k}^{2}\log\left(  M\omega\right)
}\frac{\log\left(  M\omega\right)  }{\log\omega}.
\]
However, it was shown in \cite{DH} that
\[
\limsup_{\omega\rightarrow0^{+}}\inf_{\varepsilon,N}\limsup_{k\rightarrow
\infty}\frac{\log\nu_{2}\left(  \mathcal{U}\left(  f\left(  \vec{A}%
_{k}\right)  \right)  ,M\omega\right)  }{-m_{k}^{2}\log\left(  M\omega\right)
}=\delta_{0}\left(  p\left(  \vec{x}\right)  ,\hat{\tau}\right)  =\delta
_{0}\left(  b,\hat{\tau}\right)  .
\]
Therefore we have shown that $\delta_{top}\left(  x_{1},\ldots,x_{n}\right)
\geq\delta_{0}\left(  b,\hat{\tau}\right)  $ whenever $b=b^{\ast}$ is a $\ast
$-polynomial in $\pi_{\tau}\left(  \vec{x}\right)  $. However, if $b=b^{\ast
}=\pi\left(  C^{\ast}\left(  \vec{x}\right)  \right)  ^{\prime\prime}$ is
arbitrary, it follows from the Kaplansky density theorem that there is a
sequence $\left\{  b_{k}\right\}  $ of selfadjoint elements such that each
$b_{k}$ is a $\ast$-polynomial in $\pi_{\tau}\left(  \vec{x}\right)  $ such
that $\left\Vert b_{k}\right\Vert \leq\left\Vert b\right\Vert $ and such that
$b_{k}\rightarrow b$ in the $\ast$-strong operator topology, which implies
that $\left(  b_{k},\hat{\tau}\right)  \overset{dist}{\longrightarrow}\left(
b,\hat{\tau}\right)  $. However, D. Voiculescu's semicontinuity theorem
\cite{DV3} for $\delta_{0}$ implies that $\delta_{0}\left(  b,\hat{\tau
}\right)  \leq\liminf_{k\rightarrow\infty}\delta_{0}\left(  b_{k},\hat{\tau
}\right)  \leq\delta_{0}\left(  \vec{x}\right)  $. \bigskip
\end{proof}

Voiculescu proved that if $x=x^{\ast}$ is an element of a von Neumann algebra
with faithful trace $\tau$, then%
\[
\delta_{0}\left(  x\right)  =1-%
%TCIMACRO{\dsum _{t\text{ is an eigenvalue of }x}}%
%BeginExpansion
{\displaystyle\sum_{t\text{ is an eigenvalue of }x}}
%EndExpansion
\tau\left(  P_{t}\right)  ^{2},
\]
where $P_{t}$ is the orthogonal projection onto $\ker\left(  x-t\right)  $. If
$x$ has no eigenvalues, then $\delta_{0}\left(  x\right)  =1$. It is
well-known \cite{KR} that every selfadjoint element of a finite von Neumann
algebra $\mathcal{M}$ has an eigenvalue if and only if $\mathcal{M}$ has a
finite-dimensional invariant subspace.

\begin{theorem}
\label{alt}Suppose $\mathcal{A}=C^{\ast}\left(  x_{1},\ldots,x_{n}\right)  $
is an $MF$-algebra and either

\begin{enumerate}
\item $\mathcal{A}$ has no finite-dimensional representations, or

\item $\mathcal{A}$ has infinitely many non-unitarily-equivalent
finite-dimensional irreducible representations.
\end{enumerate}
\end{theorem}

Then $\delta_{top}\left(  x_{1},\ldots,x_{n}\right)  \geq1$.

\begin{proof}
First suppose $\mathcal{A}$ has infinitely many non-unitarily-equivalent
finite-dimensional irreducible representations $\pi_{1},\pi_{2},\ldots$ . It
follows from finite-dimensionality that, for each positive integer $N$, if we
let $\rho_{N}=\pi_{1}\oplus\cdots\oplus\pi_{N},~$then
\[
\rho_{N}\left(  \mathcal{A}\right)  =\rho_{N}\left(  \mathcal{A}\right)
^{\prime\prime}=\pi_{1}\left(  \mathcal{A}\right)  \oplus\cdots\oplus\pi
_{N}\left(  \mathcal{A}\right)  ,
\]
which implies
\[
\dim\rho_{N}\left(  \mathcal{A}\right)  =\dim\left[  \pi_{1}\left(
\mathcal{A}\right)  \oplus\cdots\oplus\pi_{N}\left(  \mathcal{A}\right)
\right]  \geq N.
\]
However, it follows from Corollary \ref{pig} that
\[
\delta_{top}\left(  x_{1},\ldots,x_{n}\right)  \geq\delta_{top}\left(
\rho_{N}\left(  x_{1}\right)  ,\ldots,\rho_{N}\left(  x_{n}\right)  \right)
=1-\frac{1}{\dim\rho_{N}\left(  \mathcal{A}\right)  }\geq1-\frac{1}{N}%
\]
for $N=1,2,\ldots$ . Hence $\delta_{top}\left(  x_{1},\ldots,x_{n}\right)
\geq1.$

Next assume that $\mathcal{A}$ has no finite-dimensional representations.
Suppose $\tau\in\mathcal{T}_{MF}\left(  \mathcal{A}\right)  $. Let $\left(
\pi,M,e\right)  $ denote the GNS construction for $\tau$, i.e., $\pi
:\mathcal{A}\rightarrow B\left(  M\right)  $ is a unital $\ast$-homomorphism
with a unit cyclic vector $e$ such that, for every $a\in\mathcal{A}$, we have
$\tau\left(  a\right)  =\left(  \pi\left(  a\right)  e,e\right)  $. Let
$\mathcal{B=\pi}\left(  \mathcal{A}\right)  ^{\prime\prime}$. Since
$\mathcal{A}$ has no finite-dimensional representation, $\pi\left(
\mathcal{A}\right)  ^{\prime\prime}$ has no nonzero finite-dimensional
invariant subspace. Hence there is an $a=a^{\ast}\in\pi\left(  \mathcal{A}%
\right)  ^{\prime\prime}$ such that $a$ has no eigenvalues. Therefore from
Voiculescu's formula, $\delta_{0}\left(  a\right)  =1$. By Theorem \ref{lower}
we conclude $\delta_{top}\left(  x_{1},\ldots,x_{n}\right)  \geq1$.
\end{proof}

\bigskip

\begin{corollary}
\label{rfd2}If $\mathcal{A}$ is a unital residually finite-dimensional
C*-algebra, then, for any generating set $\left\{  x_{1},\ldots,x_{n}\right\}
$ of $\mathcal{A}$, we have%
\[
\delta_{top}\left(  x_{1},\ldots,x_{n}\right)  \geq1-\frac{1}{\dim
_{\mathbb{C}}\mathcal{A}}.
\]
If, in addition, $\mathcal{A}$ is MF-nuclear, then equality holds.\bigskip
\end{corollary}

\begin{corollary}
Suppose $\mathcal{A}$ is a unital finitely generated $MF$ C*-algebra and $G$
is a finitely generated infinite abelian group and $\alpha:G\rightarrow
Aut\left(  \mathcal{A}\right)  $ is a group homomorphism. If
$\mathcal{A\rtimes}_{\alpha}G$ is $MF$, then, for every set $\left\{
x_{1},\ldots,x_{n}\right\}  $ of generators for $\mathcal{A\rtimes}_{\alpha}%
G$, we have
\[
\delta_{top}\left(  x_{1},\ldots,x_{n}\right)  \geq1.
\]
If, in addition, $\mathcal{A}$ is MF-nuclear, then%
\[
\delta_{top}\left(  x_{1},\ldots,x_{n}\right)  =1.
\]

\end{corollary}

\begin{proof}
Since $G$ is finitely generated, $G$ is a direct sum of cyclic groups, and
since $G$ is infinite, at least one of these cyclic summands must be infinite.
Thus $G$ has generators $u_{1},\ldots,u_{m}$ with $\left\vert u_{1}\right\vert
=\infty$, and $G$ is the direct sum of the cyclic groups generated by each
$u_{k}$. If $\mathcal{A\rtimes}_{\alpha}G$ has no finite-dimensional
representations, then Theorem \ref{lower} implies $\delta_{top}\left(
x_{1},\ldots,x_{n}\right)  \geq1$. Suppose $\pi:\mathcal{A\rtimes}_{\alpha
}G\rightarrow\mathcal{M}_{d}\left(  \mathbb{C}\right)  $ is an irreducible
representation. Suppose $\theta$ is an irrational number in $\left[
0,1\right]  $. For each positive integer $k$, define a group homomorphism
$\rho_{k}:G\rightarrow\left\{  \lambda\in\mathbb{C}:\left\vert \lambda
\right\vert =1\right\}  $ by $\rho_{k}\left(  u_{1}\right)  =e^{2\pi
ik\theta/d}$ and $\rho_{k}\left(  u_{j}\right)  =1$ for $2\leq j\leq m$. We
then define a unitary group representation $\tau_{k}:G\rightarrow
\mathcal{M}_{d}\left(  \mathbb{C}\right)  $ by
\[
\tau_{k}\left(  u\right)  =\rho_{k}\left(  u\right)  \pi\left(  u\right)  .
\]
Since, for every $a\in\mathcal{A}$ and every $u\in G$, we have%
\[
\tau_{k}\left(  u\right)  \pi\left(  a\right)  \tau_{k}\left(  u\right)
^{\ast}=\pi\left(  uau^{\ast}\right)  =\alpha\left(  u\right)  \left(
a\right)  ,
\]
it follows from the defining property of the crossed product that there is a
representation $\pi_{k}:\mathcal{A\rtimes}_{\alpha}G\rightarrow\mathcal{M}%
_{d}\left(  \mathbb{C}\right)  $ such that $\pi_{k}|\mathcal{A=\pi}$ and
$\pi_{k}|G=\tau_{k}$. It is clear that the range of $\pi_{k}$ equals the range
of $\pi,$ so each $\pi_{k}$ is irreducible. Since $\mathrm{\det}\left(
\pi_{k}\left(  u_{1}\right)  \right)  =e^{2\pi ik\theta}\mathrm{\det}\left(
\pi\left(  u_{1}\right)  \right)  ,$ it follows that no two of the $\pi_{k}$'s
are unitarily equivalent. Again, from Theorem \ref{alt}, we conclude
$\delta_{top}\left(  x_{1},\ldots,x_{n}\right)  \geq1$.

If $\mathcal{A}$ is nuclear, then $\mathcal{A\rtimes}_{\alpha}G$ is nuclear,
so, by \cite{HLS}, $\delta_{top}\left(  x_{1},\ldots,x_{n}\right)  \leq1;$
whence $\delta_{top}\left(  x_{1},\ldots,x_{n}\right)  =1$.
\end{proof}

\begin{corollary}
If $\mathcal{A}$ is a simple, MF-nuclear C*-algebra, then, for any generating
set $\left\{  x_{1},\ldots,x_{n}\right\}  $ of $\mathcal{A}$, we have%
\[
\delta_{top}\left(  x_{1},\ldots,x_{n}\right)  =1-\frac{1}{\dim_{\mathbb{C}%
}\mathcal{A}}.
\]

\end{corollary}

\begin{proof}
We know the conclusion is true if $\mathcal{A}$ is finite-dimensional. If
$\mathcal{A}$ is infinite-dimensional, the simplicity of $\mathcal{A}$ implies
that $\mathcal{A}$ has no finite-dimensional representations, so, by Theorem
\ref{lower}, $\delta_{top}\left(  x_{1},\ldots,x_{n}\right)  \geq1$. Since
$\mathcal{A}$ is nuclear, we know from \cite{HLS} that $\delta_{top}\left(
x_{1},\ldots,x_{n}\right)  \leq1$ . Hence%
\[
\delta_{top}\left(  x_{1},\ldots,x_{n}\right)  =1=1-\frac{1}{\dim_{\mathbb{C}%
}\mathcal{A}}.
\]
\bigskip
\end{proof}

\begin{corollary}
Suppose $X$ is an infinite compact metric space and $\mathcal{B}$ is a
finitely generated unital MF C*-algebra. Then, for every generating set
$\left\{  x_{1},\ldots,x_{n}\right\}  $ for $C\left(  X\right)  \otimes
\mathcal{B}$, we have%
\[
\delta_{top}\left(  x_{1},\ldots,x_{n}\right)  \geq1.
\]
If, in addition, $\mathcal{B}$ is MF-nuclear, then%
\[
\delta_{top}\left(  x_{1},\ldots,x_{n}\right)  =1.
\]

\end{corollary}

\begin{proof}
If $\mathcal{B}$ has no finite-dimensional representations, then neither does
\newline$C\left(  X\right)  \otimes\mathcal{B}$. On the other hand if
$\mathcal{B}$ has an irreducible finite-dimensional representation, then since
$X$ is infinite, $C\left(  X\right)  \otimes\mathcal{B}$ has infinitely many
inequivalent irreducible finite-dimensional representations; by Theorem
\ref{alt}, $\delta_{top}\left(  x_{1},\ldots,x_{n}\right)  \geq1$.
\end{proof}

\bigskip

\begin{theorem}
If $\mathcal{A}=C^{\ast}\left(  x_{1},\ldots,x_{n}\right)  $ is an MF-nuclear
algebra and RFD, then%
\[
\delta_{\text{\textrm{top}}}\left(  x_{1},\ldots,x_{n}\right)  =1-\frac
{1}{\dim_{\mathbb{C}}\left(  \mathcal{A}/\mathcal{J}_{MF}\right)  }.
\]

\end{theorem}

\bigskip

\begin{example}
Suppose $\mathcal{B}$ is a unital separable $MF$ C*-algebra that is not
nuclear, e.g., $\mathcal{B}=C_{r}^{\ast}\left(  \mathbb{F}_{2}\right)  $, and
let $\mathcal{J}=\mathcal{B}\otimes\mathcal{K}\left(  \ell^{2}\right)  $. Then
$\mathcal{J}$ is singly generated \cite{OZ}, and every tracial state vanishes
on $\mathcal{J}$. Let $\mathcal{J}^{+}$ be the C*-algebra obtained by
adjoining the identity to $\mathcal{J}$ and suppose $\mathcal{N}$ is a
finitely generated nuclear MF C*-algebra. Then $\mathcal{A}=\mathcal{N}%
\otimes\mathcal{J}^{+}$ is finitely generated and MF, but not nuclear.
However, $1\otimes\mathcal{J}^{+}\subseteq\mathcal{J}_{MF}\left(
\mathcal{A}\right)  $, so
\[
\mathcal{J}_{MF}\left(  \mathcal{A}\right)  =\mathcal{J}_{MF}\left(
\mathcal{N}\right)  \otimes\mathcal{J}.
\]
Thus $\mathcal{A}/\mathcal{J}_{MF}\left(  \mathcal{A}\right)  $ is isomorphic
to $\mathcal{N}/\mathcal{J}_{MF}\left(  \mathcal{N}\right)  $, which is
nuclear. Hence, $\mathcal{A}$ is MF-nuclear, so for every set $\left\{
x_{1},\ldots,x_{n}\right\}  $ of generators of $\mathcal{A}$, we have%
\[
\delta_{\text{\textrm{top}}}\left(  x_{1},\ldots,x_{n}\right)  \leq1.
\]

\end{example}

\section{Special Classes of C*-algebras}

In this last section we consider two classes of C*-algebras that are important
in our work

\subsection{The Class $\mathcal{S}$}

We now consider the class $\mathcal{S}$ of separable $MF$ C*-algebras for
which every trace is an MF-trace, i.e., $\mathcal{TS}\left(  \mathcal{A}%
\right)  =\mathcal{T}_{MF}\left(  \mathcal{A}\right)  $. Recall that an $AH$
C*-algebra is a direct limit of subalgebras of finite direct sums of
commutative C*-algebras tensored with matrix algebras.

\bigskip

\begin{theorem}
The following are true for $MF$ C*-algebras $\mathcal{A}=C^{\ast}\left(
x_{1}\ldots,x_{s}\right)  $, $\mathcal{B}$ $=C^{\ast}\left(  y_{1}%
,\ldots,y_{t}\right)  $.
\end{theorem}

\begin{enumerate}
\item $\mathcal{A}\in\mathcal{S}$ if and only if every factor tracial state on
$\mathcal{A}$ is an MF-trace.

\item $\mathcal{A},\mathcal{B}\in\mathcal{S}$ if and only if $\mathcal{A}%
\oplus\mathcal{B}\in\mathcal{S}$.

\item If $\mathcal{A},\mathcal{B}\in\mathcal{S}$ and one of $\mathcal{A}$ and
$\mathcal{B}$ is exact, $\nu$ is a C*-tensor norm such that $\mathcal{A}%
\otimes_{\nu}\mathcal{B}$ is MF, then $\mathcal{A}\otimes_{\nu}\mathcal{B}%
\in\mathcal{S}$.

\item $\mathcal{A}\in\mathcal{S}$ if and only if $\mathcal{A}\otimes
\mathcal{M}_{n}\left(  \mathbb{C}\right)  \in\mathcal{S}$ for every $n\geq1$.

\item $\mathcal{S}$ is closed under direct (inductive) limits.

\item Every separable commutative C*-algebra is in $\mathcal{S}$.

\item If every factor tracial state on $\mathcal{A}$ is finite-dimensional,
and $\mathcal{B}\subseteq\mathcal{A}$, then $\mathcal{B}\in\mathcal{S}$.

\item Every $AH$ C*-algebra is in $\mathcal{S}$.

\item Every type $I$ MF C*-algebra $\mathcal{A}$ is in $\mathcal{S}$.
\end{enumerate}

\begin{proof}
$\left(  1\right)  $. It was shown in \cite{HM} that the set of factor tracial
states is the set of extreme points of $\mathcal{TS}\left(  \mathcal{A}%
\right)  $. Since $\mathcal{T}_{MF}\left(  \mathcal{A}\right)  $ is compact
and convex, $\left(  1\right)  $ follows from the Krein-Milman theorem.

$\left(  2\right)  .$ This is obvious from $\left(  1\right)  $ since if
$\tau$ is a factor tracial state, then $\pi_{\tau}\left(  1\oplus0\right)  $
is $1$ or $0$.

$\left(  3\right)  .$ Suppose $\tau$ is a factor trace on $\mathcal{A}%
\otimes_{\nu}\mathcal{B}$, and let $\left(  \pi_{\tau},H,e\right)  $ be the
GNS construction for $\tau$. Then $\pi_{\tau}\left(  \mathcal{A}\otimes_{\nu
}\mathcal{B}\right)  ^{\prime\prime}$ is a factor von Neumann algebra with the
trace $\hat{\tau}$. Since $\pi_{\tau}\left(  \mathcal{A\otimes}1\right)
^{\prime\prime}$ commutes with $\pi_{\tau}\left(  1\otimes\mathcal{B}\right)
^{\prime\prime}$, it follows that their centers are contained in the center of
$\pi_{\tau}\left(  \mathcal{A}\otimes_{\nu}\mathcal{B}\right)  ^{\prime\prime
}$. Hence, both $\pi_{\tau}\left(  \mathcal{A\otimes}1\right)  ^{\prime\prime
}$ and $\pi_{\tau}\left(  1\otimes\mathcal{B}\right)  ^{\prime\prime}$ are
factors. Let $\alpha$, $\beta,$ respectively, be the restriction of $\tau$ to
$\pi_{\tau}\left(  \mathcal{A\otimes}1\right)  ^{\prime\prime}$, $\pi_{\tau
}\left(  1\otimes\mathcal{B}\right)  ^{\prime\prime}.$ Moreover, there are
traces $\alpha\in\mathcal{TS}\left(  \mathcal{A}\right)  $ and $\beta
\in\mathcal{TS}\left(  \mathcal{B}\right)  $ such that, for all $a\in
\mathcal{A}$ and $b\in\mathcal{B}$,
\[
\tau\left(  a\otimes1\right)  =\tau\left(  \pi_{\tau}\left(  a\otimes1\right)
\right)  =\alpha\left(  a\right)  ,\text{ and}%
\]%
\[
\tau\left(  1\otimes b\right)  =\tau\left(  \pi_{\tau}\left(  1\otimes
b\right)  \right)  =\beta\left(  b\right)  .
\]
Suppose $p$ is a projection in $\pi_{\tau}\left(  \mathcal{A\otimes}1\right)
^{\prime\prime}$ and $\tau\left(  \pi_{\tau}\left(  a\otimes1\right)  \right)
=\alpha\left(  p\right)  =\frac{1}{n}$. Then there are projections
$p_{2},\ldots,p_{n}\in\pi_{\tau}\left(  \mathcal{A\otimes}1\right)
^{\prime\prime}$ such that $p+p_{2}+\cdots+p_{n}=1$ and there are partial
isometries $v_{2},\ldots,v\,_{n}\in\pi_{\tau}\left(  \mathcal{A\otimes
}1\right)  ^{\prime\prime}$ such that $v_{j}v_{j}^{\ast}=p$ and $v_{j}^{\ast
}v_{j}=p_{j}$ for $2\leq j\leq n$. It follows, for any $b\in\pi_{\tau}\left(
1\otimes\mathcal{B}\right)  ^{\prime\prime},$ that%
\[
\tau\left(  p_{j}b\right)  =\tau\left(  v_{j}^{\ast}\left(  v_{j}b\right)
\right)  =\tau\left(  v_{j}^{\ast}\left(  bv_{j}\right)  \right)  =\tau\left(
bv_{j}v_{j}^{\ast}\right)  =\tau\left(  bp\right)  =\tau\left(  pb\right)  .
\]
Hence
\[
\tau\left(  1b\right)  =\tau\left(  \left(  p+p_{2}+\cdots+p_{n}\right)
b\right)  =n\tau\left(  pb\right)  ,
\]
which implies%
\[
\tau\left(  pb\right)  =\tau\left(  p\right)  \tau\left(  b\right)  .
\]
It follows that
\[
\tau\left(  ab\right)  =\tau\left(  a\right)  \tau\left(  b\right)
\]
for every $a\in\pi_{\tau}\left(  \mathcal{A\otimes}1\right)  ^{\prime\prime}$
and every $b\in\pi_{\tau}\left(  1\otimes\mathcal{B}\right)  ^{\prime\prime}$.
Whence, on $\mathcal{A}\otimes_{\nu}\mathcal{B}$ the trace $\tau=\alpha
\otimes\beta$. Since $\mathcal{A},\mathcal{B}\in\mathcal{S}$, we know that
$\alpha$ and $\beta$ are MF-traces. It follows from part $\left(  6\right)  $
of Proposition \ref{mft} that $\tau=\alpha\otimes\beta$ is an MF-trace on
$\mathcal{A}\otimes\mathcal{B}$. It follows from statement $\left(  1\right)
$ that $\mathcal{A}\otimes\mathcal{B}\in\mathcal{S}$.

$\left(  4\right)  .$ The "only if" part follows from $\left(  3\right)  $ and
the "if" part is obvious.

$\left(  5\right)  .$ This follows from Lemma \ref{trace}.

$\left(  6\right)  .$ A factor trace on a commutative C*-algebra is
1-dimensional, and hence, by Theorem \ref{mft}, is an MF-trace.

$\left(  7\right)  $. Suppose $\tau$ is a factor trace on $\mathcal{B}$. It
follows from \cite{Longo} and \cite{Popa} that $\tau$ can be extended to a
factor state $\varphi$ on $\mathcal{A}$. Since $\varphi$ is
finite-dimensional, $\tau$ is finite-dimensional, and hence $\tau
\in\mathcal{J}_{MF}\left(  \mathcal{A}\right)  $. It now follows from part
$\left(  1\right)  $ that $\mathcal{A}\in\mathcal{S}$.

$\left(  8\right)  .$ Suppose $\mathcal{D}$ is a finite direct sum of
commutative C*-algebras tensored with matrix algebras. Since every factor
state on $\mathcal{D}$ is finite-dimensional, we know that every C*-subalgebra
of $\mathcal{D}$ is in $\mathcal{S}$. It follows from the definition of AH
algebra and statement $\left(  5\right)  $ that every AH algebra is in
$\mathcal{S}$.

$\left(  9\right)  $. Every factor representation is a direct sum of copies of
an irreducible representation. Thus every factor tracial state must be
finite-dimensional, which, by Proposition \ref{mft}, is an MF-trace. Hence
$\mathcal{A}\in\mathcal{S}$.
\end{proof}

\bigskip

\begin{corollary}
Suppose $\mathcal{A}=C^{\ast}\left(  x_{1},\ldots,x_{n}\right)  \in
\mathcal{S}$ . Then either

\begin{enumerate}
\item There is a $\tau\in\mathcal{T}_{MF}\left(  \mathcal{A}\right)  $ and an
$a\in\pi_{\tau}\left(  \mathcal{A}\right)  ^{\prime\prime}$ such that
$\delta_{0}\left(  a\right)  =1$, or

\item $\mathcal{A}/\mathcal{J}_{MF}\left(  \mathcal{A}\right)  $ is RFD.
\end{enumerate}

Therefore, either $\delta_{\text{\textrm{top}}}\left(  x_{1},\ldots
,x_{n}\right)  \geq1,$ or $d=\dim\mathcal{A}/\mathcal{J}_{MF}\left(
\mathcal{A}\right)  <\infty$ and $\delta_{\text{\textrm{top}}}\left(
x_{1},\ldots,x_{n}\right)  =1-\frac{1}{d}$.
\end{corollary}

\begin{corollary}
\label{dog}If $\mathcal{A}=C^{\ast}\left(  x_{1},\ldots,x_{n}\right)  $ is
MF-nuclear and $\mathcal{A}\in\mathcal{S}$, then%
\[
\delta_{\text{\textrm{top}}}\left(  x_{1},\ldots,x_{n}\right)  =1-\frac
{1}{\dim\mathcal{A}/\mathcal{J}_{MF}\left(  \mathcal{A}\right)  }.
\]

\end{corollary}

\begin{remark}
It seems unlikely that $\mathcal{S}$ contains every finitely generated unital
MF C*-algebra. However, we have not yet been able to construct an MF
C*-algebra with a trace that is not an MF-trace. This question is loosely
related to Connes' famous "Embedding Problem", which asks if every separably
acting finite von Neumann algebra can be tracially embedded in an ultrapower
of the hyperfinite $II_{1}$ factor. This is known to be equivalent to the
statement that, for every C*-algebra $\mathcal{A}=C^{\ast}\left(  x_{1}%
,\ldots,x_{n}\right)  $ with a tracial state $\tau$ there is a norm-bounded
sequence $\left\{  \vec{A}_{k}\right\}  $, with $\vec{A}_{k}\in\mathcal{M}%
_{m_{k}}\left(  \mathbb{C}\right)  ^{n}$ such that%
\[
\left(  \vec{A}_{k},\tau_{m_{k}}\right)  \overset{\mathrm{dist}}%
{\longrightarrow}\left(  \vec{x},\tau\right)  .
\]
Suppose Connes' Embedding Problem has a negative answer and no such sequence
$\left\{  \vec{A}_{k}\right\}  $ exists for $C^{\ast}\left(  x_{1}%
,\ldots,x_{n}\right)  $. We know from \cite{BK} that there is an MF-algebra
$\mathcal{B}=C^{\ast}\left(  y_{1},\ldots,y_{n}\right)  $ and a unital $\ast
$-homomorphism $\pi:\mathcal{B}\rightarrow\mathcal{A}$ such that $\pi\left(
y_{j}\right)  =x_{j}$ for $1\leq j\leq n$. Define a tracial state
$\rho:\mathcal{B}\rightarrow\mathbb{C}$ by $\rho=\tau\circ\pi$. If $\rho$ is
an $MF$-trace for $\mathcal{B}$, there would be a sequence $\left\{  \vec
{A}_{k}\right\}  $ with
\[
\left(  \vec{A}_{k},\tau_{m_{k}}\right)  \overset{\mathrm{dist}}%
{\longrightarrow}\left(  \vec{y},\rho\right)  .
\]
However, for any polynomial $p,$ we have%
\[
\rho\left(  p\left(  \vec{y}\right)  \right)  =\tau\left(  p\left(  \vec
{x}\right)  \right)  ,
\]
which would yield
\[
\left(  \vec{A}_{k},\tau_{m_{k}}\right)  \overset{\mathrm{dist}}%
{\longrightarrow}\left(  \vec{x},\tau\right)  .
\]
Hence $\rho$ is not an MF-trace for $\mathcal{B}$.
\end{remark}

\subsection{The Class $\mathcal{W}$}

We now want to focus on the class $\mathcal{W}$ of all separable $MF$
C*-algebras $\mathcal{A}$ such that $\mathcal{J}_{MF}\left(  \mathcal{A}%
\right)  =\left\{  0\right\}  $. The main reason is the following immediate
consequence of Corollary \ref{dog}.

\bigskip

\begin{proposition}
\label{fox}Suppose $\mathcal{A}=C^{\ast}\left(  x_{1},\ldots,x_{n}\right)  $
is MF-nuclear and $\mathcal{A\in S}\cap\mathcal{W}$. Then%
\[
\delta_{\text{\textrm{top}}}\left(  x_{1},\ldots,x_{n}\right)  =1-\frac
{1}{\dim\mathcal{A}}.
\]

\end{proposition}

\bigskip

Here are some basic properties of the class $\mathcal{W}$.

\begin{theorem}
The following are true.

\begin{enumerate}
\item If $\left\{  \mathcal{A}_{i}:i\in I\right\}  \subseteq\mathcal{W}$, and
$\mathcal{A}$ is a separable unital subalgebra of the C*-direct product $%
%TCIMACRO{\dprod _{i\in I}}%
%BeginExpansion
{\displaystyle\prod_{i\in I}}
%EndExpansion
\mathcal{A}_{i}$, then $\mathcal{A}\in\mathcal{W}$.

\item If $\mathcal{A},\mathcal{B}\in\mathcal{W}$ and one of $\mathcal{A}$ and
$\mathcal{B}$ is nuclear, then $\mathcal{A}\otimes\mathcal{B}\in\mathcal{W}$.

\item $\mathcal{A}\oplus\mathcal{B}\in\mathcal{W}$ if and only if
$\mathcal{A}\in\mathcal{W}$ and $\mathcal{B}\in\mathcal{W}$.

\item If $n\geq1$, then $\mathcal{A\in W}$ if and only if $\mathcal{M}%
_{n}\left(  \mathbb{C}\right)  \otimes\mathcal{A}\in\mathcal{W}$.

\item Every separable unital simple MF C*-algebra is in $\mathcal{W}$.

\item Every separable unital RFD C*-algebra is in $\mathcal{W}$.

\item $\mathcal{W}$ is not closed under direct limits.
\end{enumerate}
\end{theorem}

\begin{proof}
$\left(  1\right)  .$ This is a consequence of part $\left(  3\right)  $ of
Proposition \ref{mft}.

$\left(  2\right)  $. Suppose $\mathcal{A}$, $\mathcal{B}$ $\in\mathcal{W}$.
Then for every $A\in\mathcal{A}$ and every $B\in\mathcal{B}$ we have%
\[
\left\Vert A\right\Vert =\sup_{\alpha\in\mathcal{T}_{MF}\left(  \mathcal{A}%
\right)  }\left\Vert \pi_{\alpha}\left(  A\right)  \right\Vert \text{ and
}\left\Vert B\right\Vert =\sup_{\beta\in\mathcal{T}_{MF}\left(  \mathcal{B}%
\right)  }\left\Vert \pi_{\beta}\left(  B\right)  \right\Vert .
\]
It follows from part $\left(  6\right)  $ of Proposition \ref{mft} that%
\[
\left\{  \alpha\otimes\beta:\alpha\in\mathcal{T}_{MF}\left(  \mathcal{A}%
\right)  ,\beta\in\mathcal{T}_{MF}\left(  \mathcal{B}\right)  \right\}
\subseteq\mathcal{T}_{MF}\left(  \mathcal{A}\otimes\mathcal{B}\right)  .
\]
Moreover, for each such $\alpha,\beta$ we have $\pi_{\alpha\otimes\beta}%
=\pi_{\alpha}\otimes\pi_{\beta}.$ Thus
\[
\sup\left\{  \left\Vert \pi_{\alpha\otimes\beta}\left(  A\otimes B\right)
\right\Vert :\alpha\in\mathcal{T}_{MF}\left(  \mathcal{A}\right)  ,\beta
\in\mathcal{T}_{MF}\left(  \mathcal{B}\right)  \right\}  =\left\Vert
A\right\Vert \left\Vert B\right\Vert .
\]
Hence%
\[
\sup_{\tau\in\mathcal{T}_{MF}\left(  \mathcal{A}\otimes\mathcal{B}\right)
}\left\Vert \pi_{\tau}\left(  T\right)  \right\Vert
\]
is a C*-cross norm on $\mathcal{A}\otimes\mathcal{B}$, but since one of
$\mathcal{A},\mathcal{B}$ is nuclear, there is only one such norm. Hence%
\[
\left\Vert T\right\Vert =\sup_{\tau\in\mathcal{T}_{MF}\left(  \mathcal{A}%
\otimes\mathcal{B}\right)  }\left\Vert \pi_{\tau}\left(  T\right)  \right\Vert
,
\]
which implies $\mathcal{A}\otimes\mathcal{B}\in\mathcal{W}$.

Statements $\left(  3\right)  $, $\left(  4\right)  $, $\left(  5\right)  $
and $\left(  6\right)  $ are obvious.

$\left(  7\right)  .$ Let $\mathcal{A}=\mathcal{K}\left(  \ell^{2}\right)
+\mathbb{C}1$. We know that there is no nonzero continuous trace on the
algebra $\mathcal{K}\left(  \ell^{2}\right)  $ of compact operators, which
means $\mathcal{A}\notin\mathcal{W}$. However, if $\left\{  P_{n}\right\}  $
is an increasing sequence of finite-rank projections converging to $1$ in the
strong operator topology, then $\mathcal{A}$ is the direct limit of the
finite-dimensional algebras $\mathcal{A}_{n}=P_{n}\mathcal{K}\left(  \ell
^{2}\right)  P_{n}+\mathbb{C}1$.
\end{proof}

\bigskip

\bigskip

Although characterizing the class $\mathcal{W}$ may be difficult, the
following problem should be tractable in terms of Brattelli diagrams.

\begin{problem}
Which $AF$ algebras are in $\mathcal{W}$?
\end{problem}

\bigskip

Proposition \ref{fox} leads to the following semicontinuity result.

\begin{theorem}
Suppose, for each $s\geq0$, $\mathcal{A}_{s}=C^{\ast}\left(  \vec{A}%
_{s}=\left(  A_{s1},\ldots,A_{sn}\right)  \right)  $ is nuclear and in
$\mathcal{S}\cap\mathcal{W}$ and suppose $\vec{A}_{s}\overset{t.d.}%
{\longrightarrow}\vec{A}_{0}$. Then%
\[
\delta_{\text{\textrm{top}}}\left(  \vec{A}_{0}\right)  \leq\liminf
_{s\rightarrow\infty}\delta_{\text{\textrm{top}}}\left(  \vec{A}_{s}\right)
.
\]

\end{theorem}

\bigskip

Without the restriction of being in $\mathcal{S}$ in the preceding theorem,
the semicontinuity situation is not very good, even when the limit algebra is commutative.

\begin{theorem}
Suppose $n\in\mathbb{N}$ and $C^{\ast}\left(  x_{1},\ldots,x_{n}\right)  $ is
MF and has a $1$-dimensional unital representation $\alpha$. Then there is a
sequence $\left\{  \vec{A}_{s}\right\}  $ such that
\[
\vec{A}_{s}\overset{t.d.}{\longrightarrow}\vec{x}%
\]
and, for every $s\geq1$, $\delta_{\text{top}}\left(  \vec{A}_{s}\right)  =0$.
\end{theorem}

\begin{proof}
Suppose $s\in\mathbb{N}.$ Suppose $H$ is a separable Hilbert space that
contains $\mathbb{C}^{k}$ for each positive integer $k,$ and let $I_{k}$ be
the identity operator on $H\ominus\mathbb{C}^{k}.$ Since $\delta_{\text{top}%
}\left(  x_{1},\ldots,x_{n}\right)  $ is defined, there is a positive integer
$k$ and an $\vec{B}\in\mathcal{M}_{k}\left(  \mathbb{C}\right)  ^{n}$ such
that
\[
\left\vert \left\Vert p\left(  \vec{B}\right)  \right\Vert -\left\Vert
p\left(  \vec{x}\right)  \right\Vert \right\vert <\frac{1}{s}%
\]
for every $\ast$-polynomial $p\in\mathbb{P}_{s}\left(  t_{1},\ldots
,t_{n}\right)  $ (i.e., whose degree and maximum coefficient modulus do not
exceed $s$). Let $T_{j}=A_{j}\oplus\alpha\left(  x_{j}\right)  I_{k}$ for
$1\leq j\leq n.$ Then we clearly have
\[
\left\vert \left\Vert p\left(  \vec{T}\right)  \right\Vert -\left\Vert
p\left(  \vec{x}\right)  \right\Vert \right\vert <\frac{1}{s}%
\]
for every $\ast$-polynomial $p\in\mathbb{P}_{s}\left(  t_{1},\ldots
,t_{n}\right)  $. It is obvious that the set of all $\vec{S}=\left(
S_{1},\ldots,S_{n}\right)  \in\left(  \mathcal{K}\left(  H\right)
+\mathbb{C}1\right)  ^{n}\ $such that%
\[
\left\vert \left\Vert p\left(  \vec{S}\right)  \right\Vert -\left\Vert
p\left(  \vec{x}\right)  \right\Vert \right\vert <\frac{1}{s}%
\]
for every $\ast$-polynomial $p\in\mathbb{P}_{s}\left(  t_{1},\ldots
,t_{n}\right)  $ is open. It follows from \cite{PRH} there is an $\vec{A}%
_{s}\in\left(  \mathcal{K}\left(  H\right)  +\mathbb{C}1\right)  ^{n}$ such
that $\mathcal{A}_{s}=C^{\ast}\left(  \vec{A}_{s}\right)  \ $is irreducible
and
\[
\left\vert \left\Vert p\left(  \vec{A}_{s}\right)  \right\Vert -\left\Vert
p\left(  \vec{x}\right)  \right\Vert \right\vert <\frac{1}{s}%
\]
for every $\ast$-polynomial $p\in\mathbb{P}_{s}\left(  t_{1},\ldots
,t_{n}\right)  $. Clearly, $\vec{A}_{s}\overset{t.d.}{\longrightarrow}\vec{x}%
$. Since each $\mathcal{A}_{s}=C^{\ast}\left(  \vec{A}_{s}\right)  \ $is
irreducible, $\mathcal{A}_{s}=\mathcal{K}\left(  H\right)  +\mathbb{C}1$. Thus
we conclude that $\mathcal{J}_{MF}\left(  \mathcal{A}_{s}\right)
=\mathcal{K}\left(  H\right)  $ and $\mathcal{A}_{s}/\mathcal{J}_{MF}\left(
\mathcal{A}_{s}\right)  =\mathbb{C}1.$ Thus $\delta_{\text{top}}\left(
\vec{A}_{s}\right)  =0$.
\end{proof}


\begin{thebibliography}{99}                                                                                               %


\bibitem {A}W. Arveson, An invitation to C*-algebras, Graduate Texts in
Mathematics, No. 39, Springer-Verlag, New York-Heidelberg, 1976.

\bibitem {BK}B. Blackadar, E. Kirchberg, Generalized inductive limits of
finite-dimensional C*-algebras, Math. Ann. 307 (1997), no. 3, 343-380.

\bibitem {Choi}M.-D. Choi, Almost commuting matrices need not be nearly
commuting, Proc. Amer. Math. Soc. 102 (1988) 528-533.

\bibitem {D}J. Dixmier, C*-algebras, North-Holland Mathematical Library, Vol.
15. North-Holland, Amsterdam-New York-Oxford, 1977.

\bibitem {DH}M. Dost\'{a}l, D. Hadwin, An alternative to free entropy for free
group factors, International Workshop on Operator Algebras and Operator Theory
(Linfen, 2001). Acta Math. Sin. (Engl. Ser.) 19 (2003), no. 3, 419-472.

\bibitem {KD}K. Dykema, Two applications of free entropy,\ Math. Ann. 308
(1997), no. 3, 547-558.

\bibitem {EL}R. Exel, T. Loring, Finite-dimensional representations of free
product C*-algebras. Internat. J. Math. 3 (1992), no. 4, 469-476.

\bibitem {LG1}L. Ge, Applications of free entropy to finite von Neumann
algebras,\ Amer. J. Math. 119 (1997), no. 2,467--485.

\bibitem {LG2}L. Ge, Applications of free entropy to finite von Neumann
algebras,\ II. Ann. of Math. (2) 147 (1998), no.1, 143--157.

\bibitem {LG3}L. Ge, S. Popa, On some decomposition properties for factors of
type II1,\ Duke Math. J. 94 (1998), no.1, 79--101.

\bibitem {LG4}L. Ge, J. Shen, Free entropy and property T factors,\ Proc.
Natl. Acad. Sci. USA 97 (2000), no. 18, 9881--9885 (electronic).

\bibitem {LG5}L. Ge, J. Shen, On free entropy dimension of finite von Neumann
algebras,\ Geom. Funct. Anal. 12 (2002), no. 3, 546--566.

\bibitem {US}U. Haagerup, S. Thorbjornsen, A new application of random
matrices: Ext(C red(F2)) is not a group,\ Ann. of Math. (2) 162 (2005), no. 2, 711--775.

\bibitem {DH1}D. Hadwin, Free Entropy and Approximate Equivalence in Von
Neumann Algebras, Contemporary Mathematics (1998), no 228, 111-131

\bibitem {HKM}D. Hadwin, L. Kaonga, B. Mathes; Noncommutative continuous
functions, J. Korean Math. Soc. 40 (2003), no. 5, 789--830.

\bibitem {HL}D. Hadwin, W. Li, A Note on Approximate Liftings, Oper. Matrices
3 (2009), no. 1, 125--143.

\bibitem {HLS}D. Hadwin, Q. Li, J. Shen, Topological Free Entropy Dimensions
in Nuclear C*-algebras and in Full Free Products of C*-algebras, Math arXiv:0802.0281

\bibitem {HM}D. Hadwin, X. Ma, A note on free products, Operators and Matrices
2 (2008) 53-65.

\bibitem {HS1}D. Hadwin; J. Shen, Free orbit dimension of finite von Neumann
algebras, J. Funct. Anal. 249 (2007), no. 1, 75--91.

\bibitem {HS2}D. Hadwin; J. Shen, Topological free entropy dimension, Math
arXiv: math.OA/0704.0667

\bibitem {HS3}D. Hadwin; J. Shen, Topological free entropy dimension II,
Revised version

\bibitem {HS4}D. Hadwin; J. Shen, Some Examples of Blackadar and Kirchberg's
MF Algebras, Math arXiv: math.OA/0806.4712.

\bibitem {HSL}D. Hadwin, Q. Li, J. Shen; Topological Free Entropy Dimensions
in Nuclear C*-algebras and in Full Free Products of C*-algebras, Canadian J.
Math. 63 (2011) 551-590.

\bibitem {PRH}P. R. Halmos, Irreducible operators. Michigan Math. J. 15 (1968) 215--223.

\bibitem {KJ1}K. Jung, \textquotedblleft The free entropy dimension of
hyperfinite von Neumann algebras,\textquotedblright\ Trans. Amer. Math. Soc.
355 (2003), no. 12, 5053--5089 (electronic).

\bibitem {KJ2}K. Jung, \textquotedblleft A free entropy dimension
lemma,\textquotedblright\ Pacific J. Math. 211 (2003), no. 2, 265--271.

\bibitem {KJ3}K. Jung, \textquotedblleft Strongly 1-bounded von Neumann
algebras,\textquotedblright\ Math arXiv: math.OA/0510576.

\bibitem {KJ4}K. Jung, D. Shlyakhtenko, \textquotedblleft All generating sets
of all property T von Neumann algebras have free entropy dimension $\leq$
1,\textquotedblright\ Math arKiv: math.OA/0603669.

\bibitem {KR}R. V. Kadison; J. R. Ringrose, Fundamentals of the theory of
operator algebras. Vol. II. Advanced theory. Corrected reprint of the 1986
original. Graduate Studies in Mathematics, 16. American Mathematical Society,
Providence, RI, 1997. pp. i--xxii and 399--1074.

\bibitem {Lin}H. Lin, Almost commuting selfadjoint matrices and applications,
Operator algebras and their applications (Waterloo, ON, 1994/1995) 193--233,
Fields Inst. Commun., 13, Amer. Math. Soc., Providence, RI, 1997

\bibitem {Longo}R. Longo, Solution of the factorial Stone-Weierstrass
conjecture. An application of the theory of standard split W*-inclusions,
Invent. Math. 76 (1984), no. 1, 145--155.

\bibitem {DM}D. McDuff, \textquotedblleft Central sequences and the
hyperfinite factor,\textquotedblright

\ Proc. London Math. Soc. (3) 21 1970 443--461.

\bibitem {OZ}C. Olsen and W. Zame, \textquotedblleft Some C algebras with a
single generator,\textquotedblright\ Trans. of A.M.S. 215 (1976), 205-217.

\bibitem {MPD}M. Pimsner, D. Voiculescu, \textquotedblleft Imbedding the
irrational rotation C -algebra into an AF-algebra,\textquotedblright\ J.
Operator Theory 4 (1980), no. 2, 201--210.

\bibitem {Popa}S. Popa, Semiregular maximal abelian *-subalgebras and the
solution to the factor state Stone-Weierstrass problem, Invent. Math. 76
(1984), no. 1, 157--161.

\bibitem {S}S. Sakai, The theory of W*-algebras, lecture notes, Yale
University, 1962.

\bibitem {MST}M. Stefan, \textquotedblleft Indecomposability of free group
factors over nonprime subfactors and abelian subalgebras,\textquotedblright
Pacific J. Math. 219 (2005), no. 2, 365--390.

\bibitem {MST2}M. Stefan, \textquotedblleft The primality of subfactors of
finite index in the interpolated free group factors,\textquotedblright\ Proc.
Amer. Math. Soc. 126 (1998), no. 8, 2299--2307.

\bibitem {SS}S. Szarek, \textquotedblleft Metric entropy of homogeneous
spaces,\textquotedblright\ Quantum probability, 395--410, Banach Center
Publ.,43, Polish Acad. Sci., Warsaw, 1998.

\bibitem {DV1}D. Voiculescu, \textquotedblleft Circular and semicircular
systems and free product factors,\textquotedblright\ Operator algebras,
unitary representations, enveloping algebras, and invariant theory (Paris,
1989), 45--60, Progr. Math., 92, Birkhauser, Boston, MA, 1990.

\bibitem {DV2}D. Voiculescu, \textquotedblleft The analogues of entropy and of
Fisher's information measure in free probability theory II,\textquotedblright
Invent. Math., 118 (1994), 411-440.

\bibitem {DV3}D. Voiculescu, \textquotedblleft The analogues of entropy and of
Fisher's information measure in free probability theory III:The absence of
Cartan subalgebras,\textquotedblright\ Geom. Funct. Anal. 6 (1996) 172--199.

\bibitem {DV4}D. Voiculescu, \textquotedblleft Free entropy dimension $\leq$1
for some generators of property T factors of type II1,\textquotedblright\ J.
Reine Angew. Math. 514 (1999), 113--118.

\bibitem {DV5}D. Voiculescu, \textquotedblleft The topological version of free
entropy,\textquotedblright\ Lett. Math. Phys. 62 (2002), no. 1, 71--82.
\end{thebibliography}
\end{document}